\documentclass[12pt]{article}

\usepackage[utf8]{inputenc}
\usepackage[english]{babel}
\usepackage{authblk}
\usepackage{lipsum}
\usepackage{amsfonts}
\usepackage{graphicx}
\usepackage{epstopdf}
\usepackage{algorithmic}
\usepackage{amsopn}
\usepackage{amssymb}
\usepackage{amsmath}
\usepackage{amsthm}
\usepackage{xcolor}
\usepackage{lineno}
\usepackage{hyperref}
\usepackage{geometry}
\usepackage{titlesec}
\usepackage{array} 
\usepackage{multirow} 

\hypersetup{colorlinks=true, linkcolor=blue, citecolor=blue}
\titleformat{\section}{\normalfont\large\bfseries}{\thesection.}{0.5em}{}
\titleformat{\subsection}{\normalfont\normalsize\bfseries}{\thesubsection.}{0.5em}{}
\numberwithin{equation}{section}
\numberwithin{figure}{section}
\DeclareGraphicsExtensions{.pdf,.png,.jpg}
\graphicspath{{pic/}}

\newtheorem{theorem}{Theorem}[section]
\newtheorem{proposition}[theorem]{Proposition}
\newtheorem{lemma}[theorem]{Lemma}

\newtheorem{remark}[theorem]{Remark}
\newtheorem{definition}[theorem]{Definition}

\newcommand{\bnull}{\boldsymbol 0}

\renewcommand{\bf}{\boldsymbol f}
\newcommand{\bg}{\boldsymbol g}
\newcommand{\bn}{\boldsymbol n}
\newcommand{\bp}{\boldsymbol p}
\newcommand{\bpa}{\bp_{\balpha}}
\newcommand{\bq}{\boldsymbol q}
\newcommand{\bqa}{\bq_{\balpha}}
\newcommand{\bs}{\boldsymbol s}
\newcommand{\bu}{\boldsymbol u}
\newcommand{\bua}{\bu_{\balpha}}
\newcommand{\hbuae}{\widehat\bu_{\balpha_\epsilon}}
\newcommand{\dbua}{\dot{\bu}_{\balpha}}
\newcommand{\buae}{\boldsymbol{u}_{\balpha_\epsilon}}
\newcommand{\bv}{\boldsymbol v}

\newcommand{\bx}{\boldsymbol x}

\newcommand{\bH}{\boldsymbol H}
\newcommand{\bO}{\boldsymbol O}

\newcommand{\balpha}{\boldsymbol\alpha}

\newcommand{\bTheta}{\boldsymbol\Theta}
\newcommand{\btheta}{\boldsymbol\theta}
\newcommand{\bphi}{\boldsymbol\phi}

\newcommand{\bomega}{\boldsymbol\omega}
\newcommand{\bnabla}{\boldsymbol\nabla}

\newcommand{\mC}{\mathbb C}
\newcommand{\mN}{\mathbb N}
\newcommand{\mR}{\mathbb R}

\newcommand{\Oa}{\Omega_{\balpha}}
\newcommand{\Oae}{\Omega_{\balpha_\epsilon}}
\newcommand{\Oan}{\Omega_{\balpha_n}}
\newcommand{\Oas}{\Omega_{\balpha^*}}
\newcommand{\oa}{{\bomega}_{\balpha}}
\newcommand{\GD}{\Gamma_{\mathcal D}}
\newcommand{\GN}{\Gamma_{\mathcal N}}

\DeclareMathOperator{\dx}{\operatorname{d}\!\bx}
\DeclareMathOperator{\ds}{\operatorname{d}\!\bs}
\DeclareMathOperator{\C}{\mathcal{C}}
\DeclareMathOperator{\F}{\mathcal{F}}
\DeclareMathOperator{\A}{\mathcal{A}}
\DeclareMathOperator{\M}{\mathcal{M}}
\DeclareMathOperator{\Id}{\operatorname{\bold{Id}}}
\DeclareMathOperator{\bI}{\operatorname{\bold I}}
\DeclareMathOperator{\bJ}{\operatorname{\bold J}}
\DeclareMathOperator{\D}{\operatorname{\bold D}\!}
\DeclareMathOperator{\tr}{\operatorname{tr}}
\renewcommand{\d}{\operatorname{d}\!}
\renewcommand{\div}{\operatorname{div}}

\title{Optimization of the cut configuration for skin grafts}
\author[$\dagger$]{Helmut Harbrecht}
\author[$\dagger$]{Viacheslav Karnaev}
\affil[$\dagger$]{Departement Mathematik und Informatik, 
Universit\"at Basel, \newline 4051 Basel, Switzerland}
\date{\today}

\begin{document}

\maketitle

\begin{abstract}
The subject of this work is the problem of optimizing the 
configuration of cuts for skin grafting in order to improve 
the efficiency of the procedure. We consider the optimization 
problem in the framework of a linear elasticity model. We 
choose three mechanical measures that define optimality 
via related objective functionals: the compliance, the 
\(L^p\)-norm of the von Mises stress, and the area of 
the stretched skin. We provide a proof of the existence of 
solutions for each problem, but we cannot claim uniqueness. 
We compute the gradient of the objectives with respect to 
the cut configuration using shape calculus concepts. To 
solve the problem numerically, we use the gradient descent 
method, which performs well under uniaxial stretching. 
However, in more complex cases, such as multidirectional 
stretching, its effectiveness is limited due to low sensitivity 
of the functionals. To avoid this difficulty, we use a combination 
of the genetic algorithm and the gradient descent method, 
which leads to a significant improvement in the results.
\end{abstract}
\renewcommand{\linenumberfont}{\tiny\color{gray}}

\section{Introduction}
Skin grafting is a surgical procedure. Healthy skin is 
removed from one part of the body and transplanted. 
The healthy skin covers or replaces damaged or lost 
skin, for example, on the lower leg, to get a wound 
healed. The most commonly used skin graft is the
\emph{split-thickness skin graft} (SSG) which is a 
thin layer of shaved skin, see \cite{primer1,primer2}
for example.
It is called \emph{split-thickness} because only the epidermis 
and part of the dermis is shaved off, leaving part of the skin 
behind. As a result, the part left behind, called the \emph{donor 
site}, can heal on its own without needing any additional skin 
covering. 

To transplant as small a piece of skin as possible or, vice 
versa, in order to cover as large a piece of skin as possible, 
numerous parallel rows of short cuts are made in the 
healthy, harvested skin by using a suitable skin graft mesh 
device. The question we like to address in this article is: \emph{Is 
there a better pattern to enlarge the healthy piece than just making 
parallel cuts?} We tackle this problem by means of shape optimization 
\cite{DEZ,henrot2018shape,PIR,sokolowski1992introduction},
which means we are searching for an optimal layout of the cuts.
To this end, we model the piece of skin by the equations of
linear elasticity in two spatial dimensions and try to optimize
the orientation of specific, predefined cuts such that a cost 
functional is optimized. We consider different cost functionals 
when stretching the skin: We (a) minimize the von Mises stress 
which ensures a small loading of the piece of skin, (b) maximize 
the area of the skin after the stretching procedure, and (c) minimize 
the compliance of the piece of skin. We analytically compute the 
Hadamard shape gradient for either cost functional which enables 
us to perform a gradient based optimization algorithm. 

Note that the linear-elastic model is the simplest constitutive 
model which is used to represent the mechanical behaviour of 
skin and its layers, see e.g.~\cite{joodaki2018skin} for an overview 
of different skin models. Human skin is indeed not a simple object 
from the modelling point of view and its properties would be better 
described by a nonlinear hyperelasticity model. However, we have 
chosen this simple model in order to be able to compute analytical 
expressions of the shape gradients of the cost functions under 
consideration. Moreover, the numerical simulation of the equations 
of linear elasticity can easily be carried out using the finite element 
method, compare \cite{braess,brenner} for example. We use in our
computations the finite element solver FreeFem++, see 
\cite{hecht2012new}.

Our particular setup is as follows. A quadratic piece of skin is 
subdivided into regular grid cells, each of which contains a 
cut that is anchored in the center of the cell. The design variable
is the rotation angle of the cut. However, it turns out that the 
shape functionals are not very sensitive with respect to the
design variables. Moreover, we observe numerous local minima.
We have therefore decided to group the cells into blocks that 
all match, so that the skin is represented by identical copies 
of one of these blocks. Moreover, we also combined the 
gradient-based optimization with a genetic algorithm in order
to avoid to get stuck in a local minimum. Both ideas improve 
the optimization results considerably, as can be seen later 
in Section~\ref{sct:numerix}.

We finally like to mention related literature. An extensive study of 
different patterns in skin grafting is found in \cite{gupta2022biomechanical}.
Likewise, also in \cite{graft1,graft2,graft3} one can find results of 
forward simulations for predefined patterns. In contrast to these 
articles, however, we try to find the optimal pattern for skin grafting 
in case of the specific situation under consideration by an optimization
algorithm. Indeed, we are only aware of \cite{phasefield}, where also 
shape optimization has been expoited by means of a phasefield 
model in combination with a first discretize-then-optimize approach.
Nonetheless, the results therein are completely different from 
our findings as we only vary the orientation of predefined cuts.

The rest of this article is organized as follows. In 
Section~\ref{sct:modelling}, we formulate the problem under
consideration and cast it into shape optimization problems
given by three different shape functionals. We then derive 
in Section~\ref{sct:sensitivity} the related shape gradients 
and show the existence of solutions to the given shape 
optimization problems. The numerical method is introduced
in Section~\ref{sct:numerix}. Results of the numerical
optimization process are given which demonstrate the
feasibility of our approach. Finally, concluding remarks 
are stated in Section~\ref{sct:conclusio}.

\section{Problem formulation}\label{sct:modelling}
We first introduce some general notation and define the governing 
equations describing the stretched elastic body with \(N\) cuts modelled 
as a holes with given shape, size and position of the centers. Each 
hole is uniquely determined by a deviation angle. Thus, the 
configuration of the cuts is defined by \(N\) angles. We proceed 
with the formulation of the shape optimization problems for the 
unknown as the optimal configuration of the cuts. We finish this 
section by proving the existence of solutions to the optimization 
problem under consideration.

\subsection{Notation}
Let \(\Omega\subset \mR^2\) be a bounded and connected 
domain with Lipschitz-smooth external boundary \(\Gamma\) 
that is divided into two subsets \(\GD\) and \(\GN\) satisfying
\[
	|\GD|, |\GN| > 0\quad\text{such that}\quad	\Gamma = \GD\cup\GN.
\]
We further assume that there are \(N\) holes with Lipschitz-smooth 
boundaries \(\omega_{\alpha_i}, \ i=1,\dots,N\), inside the domain.
They are of the same shape and size, positioned on an equispaced 
grid, but with different angles of deviation from the abscissa axis. 
The set of all holes is denoted as 
\[
  \oa := \bigcup_{i=1}^N \omega_{\alpha_i}, 
  \quad\text{where}\quad \balpha=(\alpha_1,\dots,\alpha_N)\in[0,2\pi)^N.
\]

We shall consider an elastic body with $N$ cuts, represented 
by the domain \(\Oa:=\Omega\setminus\oa\). 
The state of the body is determined by the vector field of 
displacements \(\bu \colon \Omega\to\mR^2\). The mechanical 
properties of the body are completely characterized by the constant, 
symmetric fourth-order stiffness tensor \(\mC\), which incorporates
the material parameters: the Young modulus \(E > 0\) and the 
Poisson ratio \(-1 < \nu < 1/2\).

Throughout this article, we use the notation 
\[
	\varepsilon(\bu):=\frac{1}{2}\left(\bnabla \bu+\bnabla\bu^\top\right), 
	\quad\text{where}\quad  [\bnabla \bu]_{i,j} := \frac{\partial u_i}{\partial x_j}, \ i,j=1,2,
\]
for the deformation tensor and 
\[
	\sigma(\bu) := \mC:\varepsilon(\bu) = 2\mu\varepsilon(\bu) 
		+  \lambda\div(\bu){\bI}
\]
for the stress tensor, where \({\bI}\) is a \(2\times2\) identity 
matrix and the Lam\'e constants are
\[
	\mu=\frac{E}{2(1+\nu)} \quad\text{and} \quad
	\lambda=\frac{\nu E}{(1+\nu)(1-2\nu)}.
\]

\subsection{Governing equations}
To simplify the model, we consider the cuts as thin, elliptical 
holes $\omega_{\alpha_i}$, $i=1,\dots,N$. The main mechanical 
quantity is the displacement field \(\bu \in H^1(\Oa)^2\). It is described 
by the equations of linear elasticity, supplemented by boundary conditions. 
The elastic body is subject to the body force \(\bf \in L^2(\Oa)^2\) in the 
whole domain \(\Oa\) and the external displacements \(\bg \in L^2(\GD)^2\) 
on the part \(\GD\) of its external boundary, while the rest part \(\GN\) and 
holes boundaries \(\bomega_{\alpha}\) are unconstrained. Therefore, we
arrive at the boundary value problem
\begin{equation}\label{eq:elasticity_system}
	\left\{\;
	\begin{aligned}
		-\div\big(\sigma(\bu)\big)&= \bf \quad\text{in}\quad \Oa, \\[1ex]
		\sigma(\bu) \bn &= \bnull \quad\text{on}\quad \GN\cup\partial \bomega_{\balpha}, \\[1ex]
		\bu &= \bg \quad\text{on}\quad \GD.
	\end{aligned}
	\right.
\end{equation}
Here, \(\bn\) denotes the outward pointing unit normal 
vector on \(\partial\Oa\). A visualization of the model can be 
found in Figure \ref{fig:model}.

\begin{figure}[hbt]
	\centering
	\includegraphics[width=0.5\linewidth]{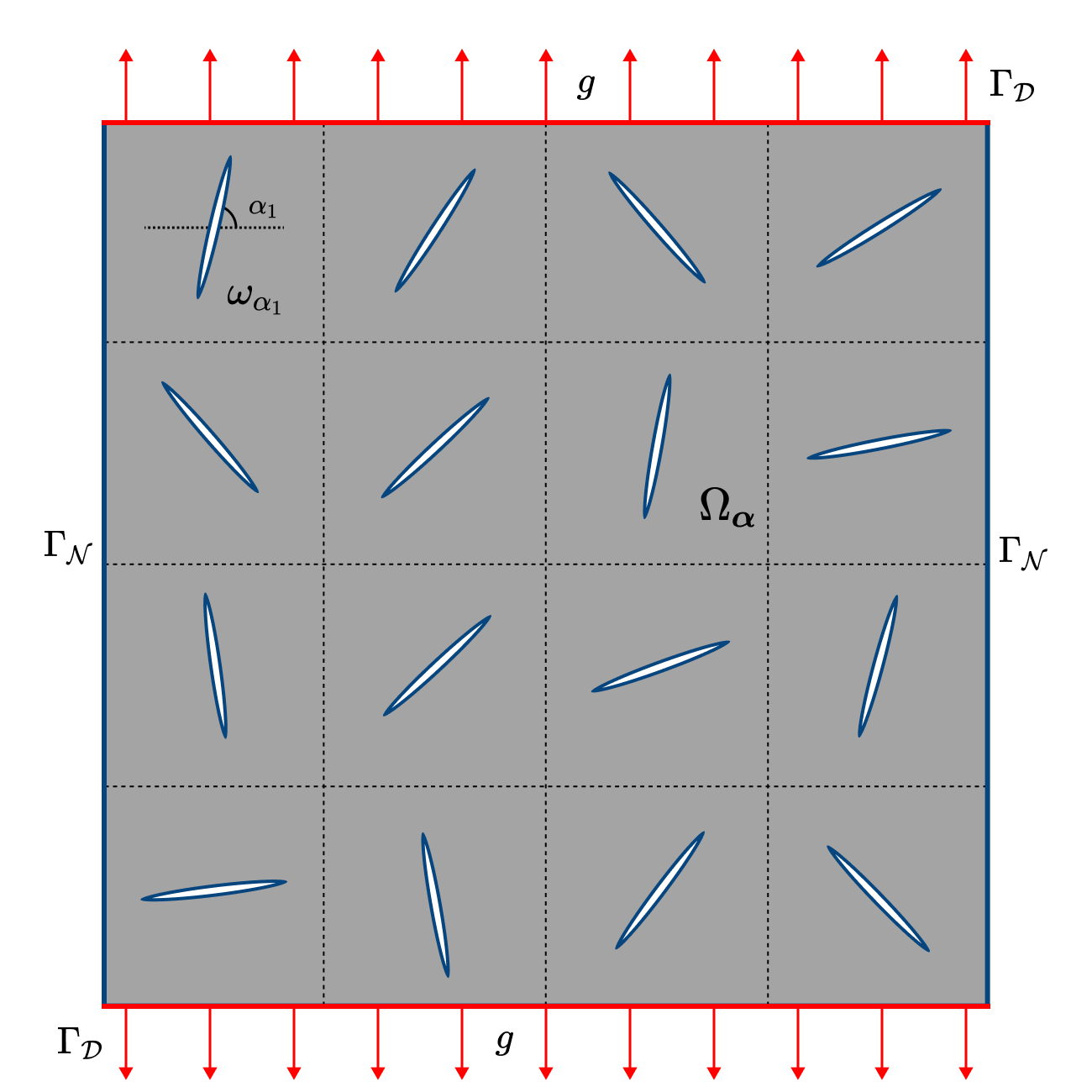}
	\caption{The model for the elastic body with cuts 
		which are represented by thin elliptical holes.}
	\label{fig:model}
\end{figure}

We shall define the space of an admissible solutions as 
\[
	\bH(\Oa) := \left\{\bu\in H^1(\Oa)^2 \ | \ \bu = \bg \ \text{on} \ \GD \right\}.
\]
Moreover, the respective test space is given by 
\[
	\bH_0(\Oa) := \left\{\bu\in H^1(\Oa)^2 \ | \ \bu = \bnull\ \text{on} \ \GD \right\}.
\]
The scalar product and norm in \(\bH(\Oa)\) defined as 
\begin{equation}\label{eq:scalar_product_norm_derfinition}
	\langle\bu,\bv\rangle_{\bH(\Oa)}
	=\int_{\Oa}\sigma(\bu): \varepsilon(\bv)\dx 
	\quad\text{and}\quad \|\bu\|_{\bH(\Oa)}
	= \int_{\Oa}\sigma(\bu): \varepsilon(\bu)\dx.
\end{equation}
In the same way it is defined in \(\bH_0(\Oa)\).
Thus, the variational formulation associated to the system 
\eqref{eq:elasticity_system} reads as follows: find \(\bu\in 
\bH(\Oa)\) such that 
\begin{equation}\label{eq:elasticity_variational_identity}
   \int_{\Oa}\sigma(\bu): \varepsilon(\bv)\dx 
   = \int_{\Oa} \bf\cdot\bv\dx \quad \forall\,\bv\in\bH_0(\Oa).
\end{equation}
It is well known that \eqref{eq:elasticity_variational_identity} 
admits a unique solution \(\bu\in\bH(\Oa)\). 

\subsection{Optimization problem}
We are looking for the optimal configuration \(\Oa\) of the model 
described above. By optimal, we mean such that the objective 
functional reaches a minimum or maximum in the set of admissible 
configurations. We consider two different cost functionals 
for the minimization problem: the compliance and the \(L^p\)-norm 
of the von Mises stress, and one for the maximization problem: 
the area of the deformed body. Here and in the following, we 
mean by \(\bua\) the solution of \eqref{eq:elasticity_system} 
for the body \(\Oa\), which determined by \(\balpha\in\mR^N\).
\begin{itemize}
	\item
	The compliance is given by 
	\[
		\C(\Oa) := \int_{\Oa}\sigma(\bua):\varepsilon(\bua)\dx.
	\]
	\item
	The von Mises stress in \(\mR^2\) is defined as 
	\[ 
		\sigma_{\text{VM}}(\bu) := \sqrt{\sigma_d(\bu):\sigma_d(\bu)},
	\]
	where \(\sigma_d\) is the stress deviator tensor 
	\[
		\sigma_d(\bu) := \sigma(\bu) - \frac{\operatorname{tr}(\sigma(\bu))}{2}\bI  
		= 2\mu\varepsilon(\bu) - \div(\bu)\bI.
	\]
	Of central interest from an applied point of view is the 
	\(L^\infty\)-norm of the von Mises stress. However, it is 
	not differentiable. Since we want to use gradient-based 
	methods, we consider the \(L^p\)-norm, which is known 
	to converge to the \(L^\infty\)-norm at \(p\to\infty\). So we 
	define the last objective function as 
	\[
		\M(\Oa) := \left(\int_{\Oa} 
		|\sigma_{\text{VM}}(\bua)|^p\dx\right)^{1/p} =  \left(\int_{\Oa} 
		(\sigma_d(\bua):\sigma_d(\bua))^{p/2}\dx\right)^{1/p}.
	\]
	\item
	The deformed body is defined as \((\Oa)_{\bua}:=(\Id + \bua)(\Oa)\), 
	where \(\Id: \mR^2\to \mR^2\) is the identity mapping. Thus, its 
	area defined as
	\[
		\A(\Oa) := \int_{(\Oa)_{\bua}}\dx = 
		\int_{\Oa}\det(\bI+\bnabla\bua)\dx .
	\]
\end{itemize}

We define the set of admissible configurations as
\[
	\bO_{2\pi} := \{\Oa\subset\mR^2 \ | \ \balpha \in (\mR \ \text{mod} \ 2\pi)^N\}.
\]
Note that the set \(\bO_{2\pi}\) is homeomorphic 
to \((\mR \ \text{mod} \ 2\pi)^N\) which is compact 
in the standard Euclidean metric \(d_{\mR^N}(\cdot,\cdot)\). 
Thus, we conclude compactness of \(\bO_{2\pi}\) in the metric 
that is defined by
\[
  d(\Oa, \Oas) := d_{\mR^N}\big(A^{-1}(\Oa),A^{-1}(\Oas)\big)
	= d_{\mR^N}(\balpha,\balpha^*),
\] 
where \(A: (\mR \ \text{mod} \ 2\pi)^N \to \bO_{2\pi}\)  
is a homeomorphism. 

Finally, we formulate the three shape 
optimization problems for finding the optimal cuts in the body 
as follows: find \(\Omega_{\balpha_{\C}}, \Omega_{\balpha_{\M}}, 
\Omega_{\balpha_{\A}}\in \bO_{2\pi}\) such that 
\begin{align}
	\C(\Omega_{\balpha_{\C}})&=\underset{\balpha\in\bO_{2\pi}}{\operatorname{inf}}
	\quad\C(\Oa), \label{eq:minimization_compliance} \\ 
	\M(\Omega_{\balpha_{\M}})&=\underset{\balpha\in\bO_{2\pi}}{\operatorname{inf}} 
	\quad\M(\Oa), \label{eq:minimization_vonmises} \\
	\A( \Omega_{\balpha_{\A}})&=\underset{\balpha\in\bO_{2\pi}}{\operatorname{sup}} 
	\quad \A(\Oa). \label{eq:minimization_area}
\end{align}

Since we formulate optimization problems on a compact set, 
we only need to show that the above functionals are continuous
in order to prove the existence of a solution. However, there is 
a certain difficulty as the functionals depend not only directly 
on \(\balpha\), but also on the solution \(\bua\). The sequence
\(\{\Oan\}_{n\in\mN}\subset\bO_{2\pi}\) corresponds to the relative 
sequence of fields \(\{\bu_n\in\bH(\Oan)\}_{n\in\mN}\), each element 
of which is defined in its own space. Therefore, for convenience, 
we will map them to a single, fixed space as follows. 

Let us fix some arbitary \(\Oas\in\bO_{2\pi}\) and a sequence 
\(\{\Oan\}_{n\in\mN}\subset\bO_{2\pi}\).  Since the cuts 
\(\omega_{\alpha_{n,i}}\) do not intersect each other  for fixed \(n\), 
we can define the non-intersecting and closed sets \(B_{r_i} := 
\{\bx\in\mR^2: \|\bx\|\leq r_i\}\) and \(B_{R_i}\{\bx\in\mR^2: 
\|\bx\|\leq R_i\}\) such that \(\omega_{\alpha_{n,i}}\subset B_{r_i}
\subsetneq B_{R_i}\subset\Omega\) for any \(n\). Let \(\chi_i(\bx)\in 
C_0^\infty(\Omega)\) be such that 
\begin{equation}\label{eq:cut_characteristic_function}
	 \chi_i(\bx) := \begin{cases} 
	 	1, \quad \text{if}\quad \bx\in  B_{r_i}, \\ 
	 	0, \quad \text{if}\quad \bx\in \Omega\setminus B_{R_i},
	 \end{cases}
\end{equation}
for \(i=1,\dots,N\). Let further \(\{\bold O_{i,n}\}_{n\in\mN}\) 
be a sequence of orthogonal matrices defined as
\[
 \bold O_{n,i} = \begin{bmatrix} \cos(\alpha_i^*-\alpha_{n,i}) 
 	& -\sin(\alpha_i^*-\alpha_{n,i}) \\ \sin(\alpha_i^*-\alpha_{n,i}) 
 	& \phantom{-}\cos(\alpha_i^*-\alpha_{n,i})  \end{bmatrix} 
 \quad\text{for any}\ i=1,\dots,N\ \text{and}\ n\in \mathbb{N}.
\]
Thus, we define a sequence \(\{\bphi_n\}_{n\in\mathbb{N}}\)  
of diffeomorphisms in accordance with
 \begin{equation}\label{eq:backmap}
	\bphi_n(\bx) := \bx + \sum_{i=1}^{N}\chi_i(\bx)(\bold O_{i,n}\bx-\bx)  
	\quad\text{for any}\ n\in \mathbb{N}.
\end{equation}
These mappings do not touch the external boundary \(\Gamma\). 
In particular, there holds \(\bphi_n(\Oan) = \Oas\) and 
\(\bphi_n^{-1}(\Oas) = \Oan\). Moreover, in the neighborhood 
of an arbitrary cut \(\omega_{\alpha_i}\), this mapping is 
expressed as \(\bphi_n(\bx)=\bold O_{i,n}\bx\) and, in the 
neighborhood of the external boundary, as \(\bphi_n(\bx)=\bx\). 
In addition, the Jacobi matrix of \(\bphi_n\) is given by
\begin{equation}\label{eq:jacobi_backmap}
	\bnabla\bphi_n(\bx) := \bI + \bJ_n(\bx)
	\quad\text{with}\quad \bJ_n(\bx) = \sum_{i=1}^{N}\bnabla
	\chi_i(\bx)\otimes(\bold O_{i,n}\bx-\bx) +\chi_i(\bx)(\bold O_{i,n} - \bI).
 \end{equation}

Before formulating the existence theorem, we need the 
following lemma and its proof.

\begin{lemma}\label{le:strong_convergence}
	Assume \(\Oas\in\bO_{2\pi}\) and consider the sequence 
	\(\{\Oan\}_{n\in\mN}\subset\bO_{2\pi}\). Let \(\bu_n := 
	\bu_{\balpha_n}\) for all $n\in\mathbb{N}$ and \(\bu^* := \bu_{\balpha^*}\) 
	denote the corresponding solutions of \eqref{eq:elasticity_variational_identity}, 
	the mapped fields \(\widehat\bu_n\) defined via the mapping \eqref{eq:backmap}.  
	If \(\Oan\to\Oas\) as \(n\to \infty\), then we have \(\widehat\bu_n := 
	\bu_n\circ\bphi_n\to \bu^*\) strongly in \(\bH(\Oas)\).
\end{lemma}

\begin{proof}
	We can express the identity \eqref{eq:elasticity_variational_identity} in the form:
	\begin{align*}
		\int_{\Oas}\mC:\big((\bI + \bJ_n)^{-1}\bnabla\widehat\bu_n\big):
		\big((\bI + \bJ_n)^{-1}\bnabla\widehat\bv\big)\det(\bI + \bJ_n)\dx \
		 = \int_{\Oas}(\bf\cdot\widehat\bv)\det(\bI + \bJ_n)\dx& \\
		  \quad \forall \bv\in \bH_0(\Oas).&
	\end{align*}
	Utilizing the Neumann series  and the properties of determinants in 
	\(\mR^2\), this can be rewritten as
	\begin{equation}\label{eq:lemmaproof_eq1}
		\begin{aligned}
		\int_{\Oas}\sigma(\widehat\bu_n): \varepsilon(\widehat\bv)\dx
		+ \mathcal{B}_n(\widehat\bu_n, \widehat\bv)
		= \int_{\Oas}\bf\cdot\widehat\bv\dx
		+ \int_{\Oas}(\bf\cdot\widehat\bv)\det(\bJ_n)\dx&\\
		\forall \bv\in \bH_0(\Oas).&
		\end{aligned}
	\end{equation}
	Herein, due to the Cauchy-Schwarz inequality and the boundedness of \(\mC\) and \(\bJ_n\),
	we have
	\[
		|\mathcal{B}_n(\bu, \bv)| \leq C \|\bold O_{i,n}
	-\bold I\|\|\bu\|_{\bH(\Oas)}\|\bv\|_{\bH(\Oas)} \quad \forall \ \bu,\bv\in \bH(\Oas),
	\]
	for a constant \(C = \text{const} > 0\).
	
	In view of \(\Oan\to\Oas\), by definition we get \(\bold O_{i,n} \to \bold I\) and 
	\(\det(\bJ_n)\to 0\) for any \(i=1,\dots,N\) as \(n\to\infty\). 
	Thus, we conclude that \(\mathcal{B}_n(\widehat\bu_n, \widehat\bv)\to 0\) 
	for any \(\widehat\bv\in \bH_0(\Oas)\) as \(n\to\infty\). 
	Hence, we get
	\begin{equation}\label{eq:lemmaproof_eq2}
		\int_{\Oas}\sigma(\widehat\bu_n): \varepsilon(\widehat\bv)\dx 
		\to  \int_{\Oas}\sigma(\widetilde\bu): \varepsilon(\widehat\bv)\dx 
		\quad \forall \ \widehat\bv\in \bH(\Oas)
	\end{equation}
	as \(n\to\infty\). Combining \eqref{eq:lemmaproof_eq1} and 
	\eqref{eq:lemmaproof_eq2}, we conclude that
	\[
		\int_{\Oas}\sigma(\widetilde\bu): \varepsilon(\widehat\bv)\dx = 
		\int_{\Oas}\bf\cdot\widehat\bv\dx \quad \forall\,\widehat\bv\in\bH_0(\Oas).
	\]
	Since the solution of \eqref{eq:elasticity_variational_identity} 
	is unique for each \(\Oa\), we arrive at \(\widetilde\bu=\bu^*\). 
	 By the definition of the scalar product \(\langle \cdot, \cdot \rangle_{\bH(\Oas)}\) 
	 in \eqref{eq:scalar_product_norm_derfinition}, and using \eqref{eq:lemmaproof_eq2}, 
	 we deduce the weak convergence \(\widehat\bu_n \rightharpoonup \bu^*\) 
	 in \(\bH(\Oas)\), i.e.,
	\begin{align*}
		\langle\widehat\bu_n,\widehat\bv\rangle_{\bH(\Oas)}
		=\int_{\Oas}\sigma(\widehat\bu_n): \varepsilon(\widehat\bv)\dx 
		\to  \int_{\Oas}\sigma(\widetilde\bu): \varepsilon(\widehat\bv)\dx 
	     =\langle\bu^*,\widehat\bv\rangle_{\bH(\Oas)}& \\
	     \quad \forall \bv\in \bH_0(\Oas).
	\end{align*}
	
	Taking the test function \(\widehat\bv = \widehat\bu_n\) in 
	\eqref{eq:lemmaproof_eq1} and exploiting the weak convergence, 
	we obtain
	\begin{equation}\label{eq:lemmaproof_eq3}
		\begin{aligned}
			\int_{\Oas}\sigma(\widehat\bu_n): \varepsilon(\widehat\bu_n)\dx
			&+ \mathcal{B}_n(\widehat\bu_n, \widehat\bu_n) \\
			&= \int_{\Oas}\bf \cdot \widehat\bu_n\dx
			+ \int_{\Oas}(\bf \cdot \widehat\bu_n)\det(\bJ_n)\dx
			\to \int_{\Oas}\bf \cdot \bu^*\dx.
		\end{aligned}
	\end{equation}
	On the other hand, we have
	\begin{equation}\label{eq:lemmaproof_eq4}
		\int_{\Oas}\bf \cdot \bu^*\dx = \int_{\Oas}\sigma(\bu^*): \varepsilon(\bu^*)\dx.
	\end{equation}
	Using the definition of the norm \(\|\cdot\|_{\bH(\Oas)}\) from \eqref{eq:scalar_product_norm_derfinition}, and combining 
	\eqref{eq:lemmaproof_eq3} with \eqref{eq:lemmaproof_eq4}, 
	we deduce the convergence
	\[
	\|\widehat\bu_n\|_{\bH(\Oas)}
	= \int_{\Oas}\sigma(\widehat\bu_n): \varepsilon(\widehat\bu_n)\dx
	\to \int_{\Oas}\sigma(\bu^*): \varepsilon(\bu^*)\dx = \|\bu^*\|_{\bH(\Oas)}.
	\]
	
	Finally, combining the weak convergence \(\widehat\bu_n \rightharpoonup \bu^*\) 
	and the  convergence \(\|\widehat\bu_n\|_{\bH(\Oas)} \to \|\bu^*\|_{\bH(\Oas)}\), 
	we conclude the strong convergence \(\widehat\bu_n \to \bu^*\) in \(\bH(\Oas)\).
\end{proof}

\begin{theorem}
	Each of the problems 
	\eqref{eq:minimization_compliance}--\eqref{eq:minimization_area}
	allows its own solution.
\end{theorem}

\begin{proof}
	For the sake of brevity, we present the proof only for \(\C(\balpha)\). 
	In complete analogy, we can get the result for \(\M\) and \(\A\), 
	respectively.
	
	Let \(\{\Oan\}_{n\in\mN}\subset\bO_{2\pi}\) and \(\Oas\in\bO_{2\pi}\) such that 
	\(\Oan\to\Oas\) as \(n\to \infty\). 
	Using the mappings defined in \eqref{eq:backmap}, 
	we rewrite the functional  \(\C(\Oan)\) as an integral 
	over the domain \(\Oas\) as 
	\begin{align*}
	\C(\Oan) = \int_{\Oan}\sigma(&\bu_n): \varepsilon(\bu_n)\dx \\
	&= \int_{\Oas}\mC:\big((\bI + \bJ_n)^{-1}\bnabla\widehat\bu_n\big):
	\big((\bI + \bJ_n)^{-1}\bnabla\widehat\bu_n\big)\det(\bI + \bJ_n)\dx. 
	\end{align*}
	From the Lemma \ref{le:strong_convergence}, we thus conclude that 
	\begin{align*}
	\C(\Oan) = \int_{\Oas}\mC:\big((\bI + \bJ_n)^{-1}\bnabla\widehat\bu_n\big):
	\big((\bI + \bJ_n&)^{-1}\bnabla\widehat\bu_n\big)\det(\bI + \bJ_n)\dx \\ 
	&\to \int_{\Oas}\sigma(\bu^*): \varepsilon(\bu^*)\dx =  \C(\Oas).
	\end{align*}
	Hence, the claim follows from the generalized extreme value theorem
	since \(\C\) is a continuous functional.
\end{proof}

\begin{remark}
	We cannot say anything about the uniqueness of the minimizer
	in each case. Especially, the shape optimization problem may 
	have a large number of local minimizers.
\end{remark}

\section{Sensitivity analysis}
\label{sct:sensitivity}
The most common way to solve optimization problems 
numerically is to use gradient based algorithms, for example 
gradient descent or the quasi-Newton method. To define the gradient 
we use the concept of shape calculus. We begin with an introduction 
of the basic notations and definitions, the finish with the expression 
of the shape gradient for each of the considered functionals. 
For a general overview of shape calculus, we refer the reader, 
for example, to \cite{allaire2021shape, henrot2018shape, 
plotnikov2023geometric, sokolowski1992introduction}.

\subsection{Fundamentals of shape calculus}
Let us consider an arbitrary shape functional \(\F: \bO_{2\pi}\to \mR\) 
and perturbation fields \(\btheta_i\in W^{1,\infty}(\Oa)^2\) 
for \(i=1,\dots,N\) such that \(\btheta_i(\bx) = 0\) if \(\bx\in 
\Gamma\cup\partial(\bomega_{\balpha}\setminus\omega_i)\), 
i.e., it only perturbs the boundary of the respective cut. Then,
for a sufficiently small parameter \(\epsilon>0\), we define 
the perturbed body as 
\[
	\Omega_{\balpha_{\epsilon_i}} := (\Id + \epsilon\btheta_i)(\Oa).
\]

Since we consider cuts of fixed shape and size, the only thing 
we can perturb is their angles. So we define \(\btheta_i\) as 
\begin{equation}\label{eq:pertub_field}
	\btheta_i(\bx) := \chi_i(\bx)\begin{bmatrix} -x_2 \\ 
		\phantom{+}x_1 \end{bmatrix},
\end{equation}
where \(\chi_i\) defined in \eqref{eq:cut_characteristic_function}. 
The whole perturbation in the neighborhood of the cut is 
expressed as
\[
	(\Id + \epsilon\btheta_i)(\bx) =  
	\begin{bmatrix} 1 & -\epsilon \\ \epsilon & \phantom{+}1 \end{bmatrix} 
	\begin{bmatrix} x_1 \\ x_2 \end{bmatrix} 
	\quad\text{if}\quad\bx\in B_{r_i},
\]
which define a linearized \(\epsilon\) angle rotation, see 
\cite{karnaev2022optimal} for the linearization. In the rest 
of the domain, the perturbation is an identical mapping
\[
	(\Id + \epsilon\btheta_i)(\bx) = \Id \quad\text{if}\quad\bx\in 
	\Omega\setminus B_{R_i}.
\]
We refer to Figure~\ref{fig:perturbation} for an illustration.

\begin{figure}[hbt]
	\centering
	\includegraphics[width=0.8\linewidth]{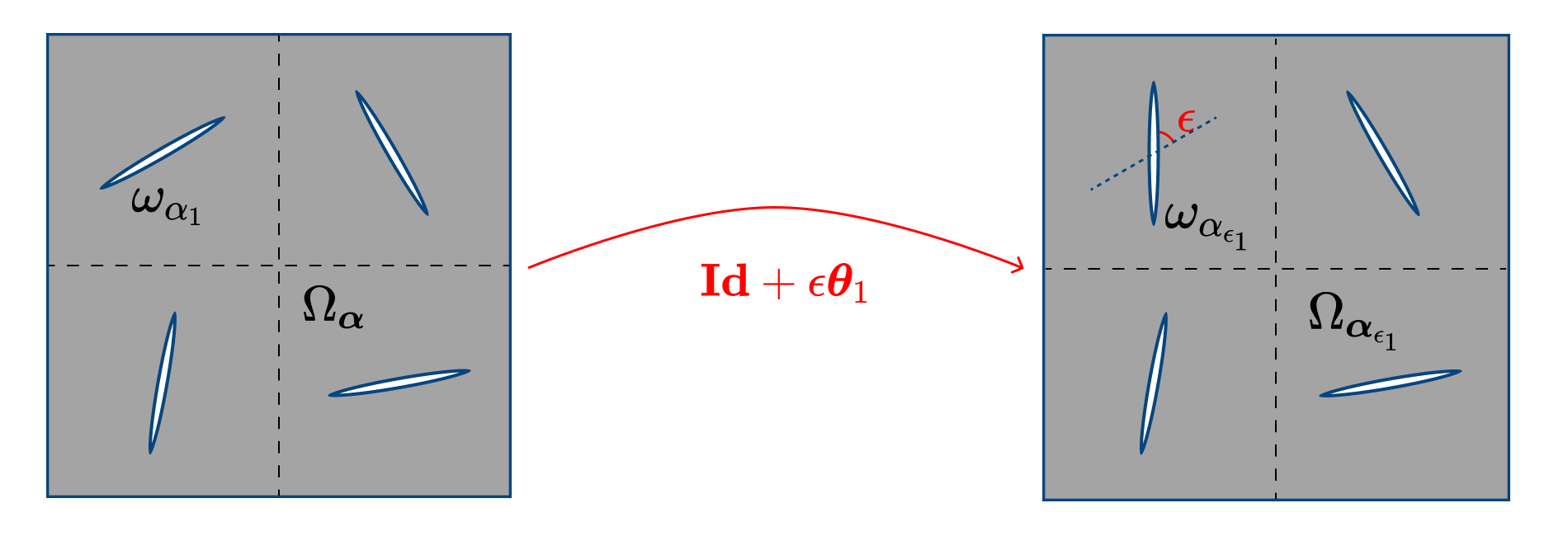}
	\caption{
		Variation \(\Omega_{\balpha_{\epsilon_i}}\)of a shape 
		\(\Oa\) according to a deformation field \(\btheta_1\).
	}
	\label{fig:perturbation}
\end{figure}

We can define the derivative of the functional along 
the direction given by the perturbation as follows.

\begin{definition}\label{def:shape_derivative}
	A \textbf{shape derivative} of the shape functional \(\F\) 
	is defined as Gateaux derivative at \(\Oa\) in direction 
	of the vector field $\btheta_i$, $i=1,\dots,N$, i.e.,
	\[
		\d\F(\Oa)\langle\btheta_i\rangle 
		:=  \frac{\d}{\d\epsilon}\F((\Id + \epsilon\btheta_i)(\Oa))\Big |_{\epsilon=0}
	\]
	Thus, we define a \textbf{shape gradient} as 
	\[
		\D\F(\Oa)\langle\bTheta\rangle := 
		\big[\d\F(\Oa)\langle\btheta_1\rangle, 
		\dots, \d\F(\Oa)\langle\btheta_N\rangle\big]^\top
		\quad\text{where}\quad \bTheta = [\btheta_1, \dots, \btheta_N].
	\]
\end{definition}

In addition, we introduce the derivative of the displacement 
field \(\bua\), which quantifies the sensitivity of the solution 
to \eqref{eq:elasticity_system} with respect to variations in 
the domain. We follow the Lagrangian approach and 
introduce the next definition.

\begin{definition}\label{def:lagrangian_derivative}
	A \textbf{Lagrangian derivative} denoted by 
	\(\dbua\langle\btheta_i\rangle\), \(i=1,\dots,N\), 
	and define by
	\[
		\dbua\langle\btheta_i\rangle
		:= \frac{\d}{\d\epsilon}\big(\bu_{\balpha_{\epsilon_i}}\circ
		(\Id+\epsilon\btheta_i)\big)\Big |_{\epsilon=0},
	\]
	where \(\bu_{\balpha_{\epsilon_i}}\) and \(\bua\) are 
	the solutions of \eqref{eq:elasticity_system} in the 
	domains \(\Omega_{\balpha_{\epsilon_i}}\) and \(\Oa\),
	respectively.
\end{definition}

In the context of unconstrained shape optimization, the shape
derivative is used to identify a direction \(\btheta\) of deformation 
such that \(\d\F(\Oa)(\btheta)< 0\). This direction of deformation 
serves as a descent direction in an appropriate optimization algorithm, 
which allows the minimization of the objective function \(\F(\Oa)\). 
Before proceeding to the shape derivatives of the functionals
under consideration, we need to note that the shape derivative 
can be expressed via the integral on the boundary of the domain.

\begin{remark}\label{rem:hadamard_form}
	In the case of a sufficiently regular domain \(\Oa\), 
	due to Hadamard's structure theorem (see e.g.\ 
	\cite{henrot2018shape,sokolowski1992introduction}),
	we can conclude that the value of the derivative 
	\(\d\F(D)\langle\cdot\rangle\) depends only on the normal 
	component of the vector field \(\btheta\) at the boundary 
	\(\partial\Oa\), i.e.,
	\[
		\d\F(\Oa)\langle\btheta\rangle =  \int_{\partial\Oa}v_{\Oa}(\btheta\cdot\bn)\ds.
	\]
	Here, \(v_{\Oa}:\partial\Oa\to\mR\) is a scalar field whose 
	expression depends on the solutions of the underlying
	boundary value problem and the functional form. The 
	expression is commonly called the \textbf{Hadamard's form} 
	of the shape derivative.
\end{remark}
\subsection{Shape derivatives}
Next follow statements about the system for the Lagrangian 
derivative of the displacement field \(\bua\) and about 
expressions for the shape derivatives of functionals \(\C(\Oa)\) 
and \(\M(\Oa)\). We do not give proofs for these statements
as they are derived by standard techniques as found for 
example in \cite{allaire2021shape}.

The next lemma characterizes the Lagrangian derivative 
of the displacement field \(\bua\).

\begin{lemma}\label{le:lagrangian_derivative}
	The Lagrangian derivative \(\dot\bu\in H^1(\Oa)^2\) 
	of the solution to \eqref{eq:elasticity_system} satisfies
	\begin{equation}\label{eq:lagrangian_derivative_system}
		\left\{\;
		\begin{aligned}
			&\div(\sigma(\dbua\langle\btheta_i\rangle) = 
			\div\Big((\mC : \bnabla\bua)\bnabla\btheta^\top + 
			\mC:(\bnabla\btheta_i\bnabla\bua)&   \\[1ex]  
			&\hspace{66mm} - \div(\btheta_i)\,\sigma(\bua)
			 - \bf\otimes\btheta_i\Big)\quad\text{in}\quad \Oa, \\[1ex]
			&\sigma(\dbua)\bn =  \Big((\mC : \bnabla\bua)\bnabla\btheta^\top 
			+ \mC:(\bnabla\btheta\bnabla\bua) - \div(\btheta_i)\sigma(\bua)\Big)\bn
			\quad\text{on}\quad \partial\omega_i, \\[1ex]
			&\sigma(\dbua)\bn =  \bnull  \quad\text{on}
			\quad \GN\cup\partial(\bomega_{\balpha}\setminus\omega_i), \\[1ex]
			&\dbua = \bnull \hspace{13mm}\text{on}\quad\GD,
		\end{aligned}
		\right.
	\end{equation}
	for any \(i=1,\dots,N\).
	An equivalent variational formulation reads as
	\begin{equation}\label{eq:lagrangian_derivative_variational_identity}
	\begin{aligned}
		\int_{\Oa}\sigma(\dbua\langle\btheta_i\rangle) :\varepsilon(\bv) = \int_{\Oa}\big(\mC:(\bnabla\btheta_i\bnabla\bua) : \bnabla\bv 
		+ \mC:\bnabla\bua:(\bnabla\btheta_i\bnabla\bv)\big)\dx& \\ 
		-\int_{\Oa}\big(\div(\btheta_i)\sigma(\bua) : \varepsilon(\bv) - \div(\bf\otimes\btheta_i)\cdot\bv\big)\dx& \\
		 \forall\,\bv\in\bH_0(\Oa)&.
	\end{aligned}
	\end{equation}
\end{lemma}

The shape derivative of the compliance \(\C(\Oa)\) 
is presented in the next proposition.

\begin{proposition}\label{pr:comopliance_shape_derivative}
	The shape derivative of \(\C(\Oa)\) is given by
	\begin{equation}\label{eq:compliance_shape_derivative_expression}
		\d\C(\Oa)\langle\btheta_i\rangle = \int_{\partial\omega_i} \big(\sigma(\bua):\varepsilon(\bua)\big)
		\big(\btheta_i\cdot\bn\big)\ds,
	\end{equation}
	where \(i=1,\dots,N\).
\end{proposition}

The shape derivative of the \(L^p\)-norm of the von 
Mises stress \(\M(\Oa)\) is given in the following 
proposition, see also \cite{caubet2023shape}.

\begin{proposition}\label{pr:vonmises_derivative}
	The shape derivative \(\M(\Oa)\) is given by
	\begin{equation}\label{eq:vonmises_derivative_expression}
		\begin{aligned}
			\d\M(\Oa)\langle\btheta_i\rangle = \frac{1}{p}\bigg(\int_{\omega_i} 
			\big((&\sigma_d(\bua):\sigma_d(\bua))^{p/2} - \sigma(\bpa):\varepsilon(\bua) \\
			&+ (\bf\cdot\bpa)\big)\big(\btheta\cdot\bn\big)\ds\bigg)\left(\int_{\Oa} 
			(\sigma_d(\bua):\sigma_d(\bua))^{p/2}\dx\right)^{1/p - 1},
		\end{aligned}
	\end{equation}
	where \(i=1,\dots,N\) and \(\bpa\in H^1(\Oa)^2\) satisfies the following adjoint system
	\begin{equation}\label{eq:vonmises_adjoint_system}
		\left\{
		\begin{aligned}
			&\div(\sigma(\bpa)) = 2p\mu\div\Big(\sigma_d(\bua)(\sigma_d(\bua):\sigma_d(\bua))^{p/2-1}\Big)  \quad\text{in}\quad\Oa, \\[1ex]
			&\sigma(\bpa)\bn = 2p\mu\Big(\sigma_d(\bua)(\sigma_d(\bua):\sigma_d(\bua))^{p/2-1}\Big) \bn \quad\text{on}\quad\GN\cup\partial \bomega_{\balpha}, \\[1ex]
			&\bpa = \bnull \quad\text{on}\quad \GD \\[1ex]
		\end{aligned}
		\right.
	\end{equation}
	with an equivalent variational formulation
	\begin{equation*}\label{eq:vm_adjoint_variational_identity}
			\int_{\Oa}\sigma(\bpa) :\varepsilon(\bv) = \int_{\Oa}p(\sigma_d(\bua):\sigma_d(\bua))^{p/2-1}\sigma_d(\bua) : \sigma_d(\bv) \quad \forall\,\bv\in\bH_0(\Oa).
	\end{equation*}
\end{proposition}

Finally, the next proposition gives the shape derivative of 
the area of the deformed body \(\A(\Oa)\). Since we have 
not seen an expression for this before, we shall provide 
the proof.

\begin{proposition}\label{pr:area_shape_derivative}
	The shape derivative of \(V(\Oa)\) is given by
	\begin{equation}\label{eq:area_derivative_expression}
		\d\A(\Oa)\langle\btheta_i\rangle =
		\int_{\omega_i}\big(\det(\bI + \bnabla\bua)  
		-\sigma(\bua) : \varepsilon(\bqa)
		+ \bf\cdot\bqa\big)\big(\btheta\cdot\bn\big)\dx,
	\end{equation}
	where \(i=1,\dots,N\) and  \(\bqa\in H^1(\Oa)^2\) satisfies the following adjoint system
	\begin{equation}\label{eq:area_derivative_adjoint_system}
		\left\{\;
		\begin{aligned}
			&\div(\sigma(\bqa)) = \div\big(\det(\bI+\bnabla\bua)(\bI+\bnabla\bua)^{-\top}\big) \quad\text{in}\quad\Oa, \\[1ex]
			&\sigma(\bqa)\bn =  \det(\bI+\bnabla\bua)(\bI+\bnabla\bua)^{-\top}\bn \quad\text{on}\quad\GN\cup\partial \bomega_{\balpha}, \\[1ex]
			&\bqa = \bnull \quad\text{on}\quad \GD \\[1ex]
		\end{aligned}
		\right.
	\end{equation}
	with an equivalent variational formulation
	\begin{equation*}\label{eq:vol_adjoint_variational_identity}
		\int_{\Oa}\sigma(\bqa) :\varepsilon(\bv) = \int_{\Oa}\det(\bI+\bnabla\bua)(\bI+\bnabla\bua)^{-\top} : \bnabla\bv \quad \forall\,\bv\in\bH_0(\Oa).
	\end{equation*}
\end{proposition}

\begin{proof}
	Since the proof does not depend on the choice of \(1\leq i \leq N\), 
	we will simplify our notation in accordance with 
	\(\btheta \equiv \btheta_i, \ \omega_i\equiv\omega,\ 
	\Oae \equiv\Omega_{\balpha_{\epsilon_i}}\) and for 
	likewise \(\buae\equiv\bu_{\balpha_{\epsilon_i}}\).
		
	For sufficiently small \(\btheta\in W^{1,\infty}(\Oa)^2\), 
	a change of variables in the shape functional \(\A(\Oae)\) yields
	\begin{align*}
		\A(\Oae) = \int_{(\Oae)_{\buae}}\dx &=  \int_{\Oae}\det(\bI+\bnabla\buae)\dx \\ 
		&= \int_{\Oa}\det(\bI+(\bI+\epsilon\bnabla\btheta)^{-1}\bnabla{\hbuae})
		\det(\bI+\epsilon\bnabla\btheta)\dx.
	\end{align*}
	By taking the derivative at \(\epsilon=0\), we obtain 
	\begin{equation}\label{eq:area_derivative_proof_differentiation}
		\begin{aligned}
			\d\A(\Oae)\langle\btheta\rangle &= \int_{\Oa}\div(\btheta)\det(\bI + \bnabla\bua)\dx \\
			&~~~~~~~~~~~~~+ \int_{\Oa}\det(\bI+\bnabla\bua)\tr(\bI+\bnabla\bua)^{-1}(\bnabla\dbua-\bnabla\btheta\bnabla\bua))\dx
			\\
			&=\int_{\Oa}\div(\btheta)\det(\bI + \bnabla\bua)\dx \\
			&~~~~~~~~~~~~~+ \int_{\Oa}\det(\bI+\bnabla\bua)(\bI+\bnabla\bua)^{-\top}:\bnabla\dbua\dx \\
			&~~~~~~~~~~~~~-\int_{\Oa}\det(\bI+\bnabla\bua)(\bI+\bnabla\bua)^{-\top}:(\bnabla\btheta\bnabla\bua)\dx.
		\end{aligned}
	\end{equation}
	
	Next, we formulate the variational identity for 
	the adjoint state \(\bqa\in \bH_0(\Oa)\) as follows:
	\begin{equation}\label{eq:area_derivative_adjoint_identity}
			\int_{\Oa} \sigma(\bqa):\varepsilon(\bv)\dx 
			=  \int_{\Oa}\det(\bI+\bnabla\bua)(\bI+\bnabla\bua)^{-\top}:\bnabla\bv\dx \quad \forall\,\bv\in\bH_0(\Oa).
	\end{equation}
	It is easy to check that the variational formulations \eqref{eq:area_derivative_adjoint_identity} 
	corresponds to the boundary value problem \eqref{eq:area_derivative_adjoint_system}. 
	By taking \(\dbua\) as test function in \eqref{eq:area_derivative_adjoint_identity} 
	and \(\bqa\) as test function in \eqref{eq:lagrangian_derivative_variational_identity}, 
	we conclude
	\begin{align*}
		\int_{\Oa}\det(\bI+\bnabla\bua)(\bI+&\bnabla\bua)^{-\top}:\bnabla\dbua\dx \\
		&= \int_{\Oa}\big(\mC:(\bnabla\btheta\bnabla\bua) : \bnabla\bqa + \mC:\bnabla\bua:(\bnabla\btheta\bnabla\bqa)\big)\dx \\
		&~~~~~~~-\int_{\Oa}\big(\div(\btheta)\sigma(\bua) : \varepsilon(\bqa) - \div(\bf\otimes\btheta)\cdot\bqa\big)\dx.
	\end{align*}
	Thus, we can rewrite \eqref{eq:area_derivative_proof_differentiation} as 
	\begin{equation}\label{eq:area_derivative_proof_volform}
		\begin{aligned}
			\d\A(\Oae)\langle\btheta\rangle &=\int_{\Oa}\div(\btheta)\det(\bI + \bnabla\bua)\dx \\
			&~~~~~~~~~~~~~+ \int_{\Oa}\big(\mC:(\bnabla\btheta\bnabla\bua) : \bnabla\bqa + \mC:\bnabla\bua:(\bnabla\btheta\bnabla\bqa)\big)\dx \\
			&~~~~~~~~~~~~~-\int_{\Oa}\big(\div(\btheta)\sigma(\bua) : \varepsilon(\bqa) - \div(\bf\otimes\btheta)\cdot\bqa\big)\dx \\
			&~~~~~~~~~~~~~-\int_{\Oa}\det(\bI+\bnabla\bua)(\bI+\bnabla\bua)^{-\top}:(\bnabla\btheta\bnabla\bua)\dx.
		\end{aligned}
	\end{equation}

	Following Remark \ref{rem:hadamard_form}, we can 
	consider that \(\btheta = (\theta\cdot\bn)\bn\) on \(\partial\omega\). 
	Therefore, by applying Green's formula to the expression
	\eqref{eq:area_derivative_proof_volform}, we arrive at
	Hadamard's form of \(\d\A(\Oa)\langle\btheta\rangle\)
	that was claimed to be proven.  
\end{proof}

\section{Numerical realization}\label{sct:numerix}
In this part, we develop the numerical method for solving the 
shape optimization problem under consideration. We employ 
both, the gradient descent method as such and its combination 
with a genetic algorithm. The intend to optimize a quadratic 
piece of skin, in which the cuts are modelled as slits with 
locations in periodic cells within the domain. We present 
the outcomes of the gradient descent method in the case 
of stretching the skin in one axis direction and the outcomes
of the combination of a genetic algorithm and the gradient 
descent method in the case of stretching in two axes 
directions.

\subsection{Optimization algorithm}
To solve the optimal cut layout problems 
\eqref{eq:minimization_compliance}--\eqref{eq:minimization_area}, 
we employ a classical gradient descent method
\begin{equation}\label{eq:grad_descent}
	\balpha_{n+1} = \balpha_{n} - 
	\gamma_n\D\F(\Oan)\langle\bTheta\rangle, 
\end{equation}
where \(\mathcal{F}\) is one of the considered functionals, 
\(\D\F(\Oan)\langle\bTheta\rangle\) is the shape gradient 
calculated in accordance with the Proposition
\ref{pr:comopliance_shape_derivative} for the related shape 
functional and the perturbation filed \(\bTheta = [\btheta_1, 
\dots, \btheta_N]\) defined by \eqref{eq:pertub_field}. The 
step size \(\gamma_n\) is found by a quadratic line search 
with at most five iterations. Note that in the case of maximization 
of the area of the deformed body, we just switched the sign 
in front of the functional to cast it into a the minimization 
problem.

Numerical experiments have shown that the shape optimization
problem under consideration has a low sensitivity and many local 
minima, which in some configurations makes the solution of the 
problem by the gradient descent method significantly more difficult, 
even when using accelerated variants such as the quasi-Newton 
or the Nesterov scheme, see \cite{fletcher,nesterov} for example. 
Therefore, inspired by the successful results of combining global 
optimization methods with continuous optimization in 
\cite{harbrecht2016optimization} for a similar shape optimization 
problem, we shall address the application of a genetic algorithm, 
cf.~\cite{haupt2004practical}.

A genetic algorithm is a stochastic method used to solve 
optimization problems. It lacks the precision of gradient-based 
methods because it does not study the function to be minimized. 
It evaluates the objective functional ({\it fitness}) for the optimization 
variables ({\it individuals}). The algorithm's process is as follows: 
an initial population of individuals evolves over multiple generations 
through simulated genetic operations of {\it mutation} and 
{\it crossover}. The fittest individuals survive and reproduce, 
advancing the population. The initial population consists of \(M\) 
individuals, obtained by the specified heuristics or generated 
randomly. To transition from one generation to the next, we follow 
these steps:
\begin{enumerate}
	\item The current population ({\it parents}) mutates and 
	produces a generation of {\it children}. 
	\item The {\it elite} are chosen from the generations of 
	parents and children with subject to fitness.
	\item The elite are crossed to produce the new children.
	\item New collected elite replaces the old population.
\end{enumerate}
As a mutation, we use a gradient descent step \eqref{eq:grad_descent} 
with step size \(\gamma_n=\|\D\F(\Oan)\langle\bTheta\rangle\|_2^{-1}\). 
Crossover children are created by intersecting the vectors of a pair of 
parents: we randomly choose two parents (here: \(p_1\) and \(p_2\)) 
and a cut point, then we exchange the components of both parents 
and create two children (here: \(c_1\) and \(c_2\)). For example:
\[
\begin{array}{l @{\qquad} c @{\qquad} l}
	p_1 = [1,2,3,4,5,6,7] & \multirow{2}{*}{$\longrightarrow$} & c_1 = [1,2,3,4,e,f,g] \\
	p_2 = [a,b,c,d,e,f,g] && c_2 = [a,b,c,d,5,6,7]
\end{array}
\]
Non-unique individuals are not allowed in the population. 
Then, starting from some stage of the fittest individual does 
not change and displaces all others in the population by 
the crossover and elite choice. When it displaces all the 
rest, the algorithm naturally stops. Note that the most 
computationally expensive part of the above algorithm is 
the computation of the solution of the direct problem and 
the calculation the shape gradient. This, however, can be
easily parallelized. 

\subsection{Computational setup}
The domain \(\Omega = (0,1)^2\), which represents the
piece of skin, is chosen as the unite square. We divide 
the domain into 3\(\times\)3 equal blocks, each of which 
being a copy of the others. Inside these blocks, we introduce
a grid of 4\(\times\)4 quadratic cells of edge size \(1/12\). 
In the center of each cell, we put an ellipse-shaped cut with 
semi-axes \(0.75/12\) and \(0.05/12\). Thus, we get \(\oa\) 
with 144 cuts in all, but only \(N=16\) design variables (the 
rotation angles of the cuts), compare Figure~\ref{fig:model}. 
We do so in order to improve the optimization results, since 
each of the components \(\balpha\) defines nine cuts at once 
which increases the sensitivity of the functional with respect 
to a change in the design variables. 

\begin{figure}[hbt]
\centering
\includegraphics[width=0.6\linewidth,trim={150 150 150 150},clip]{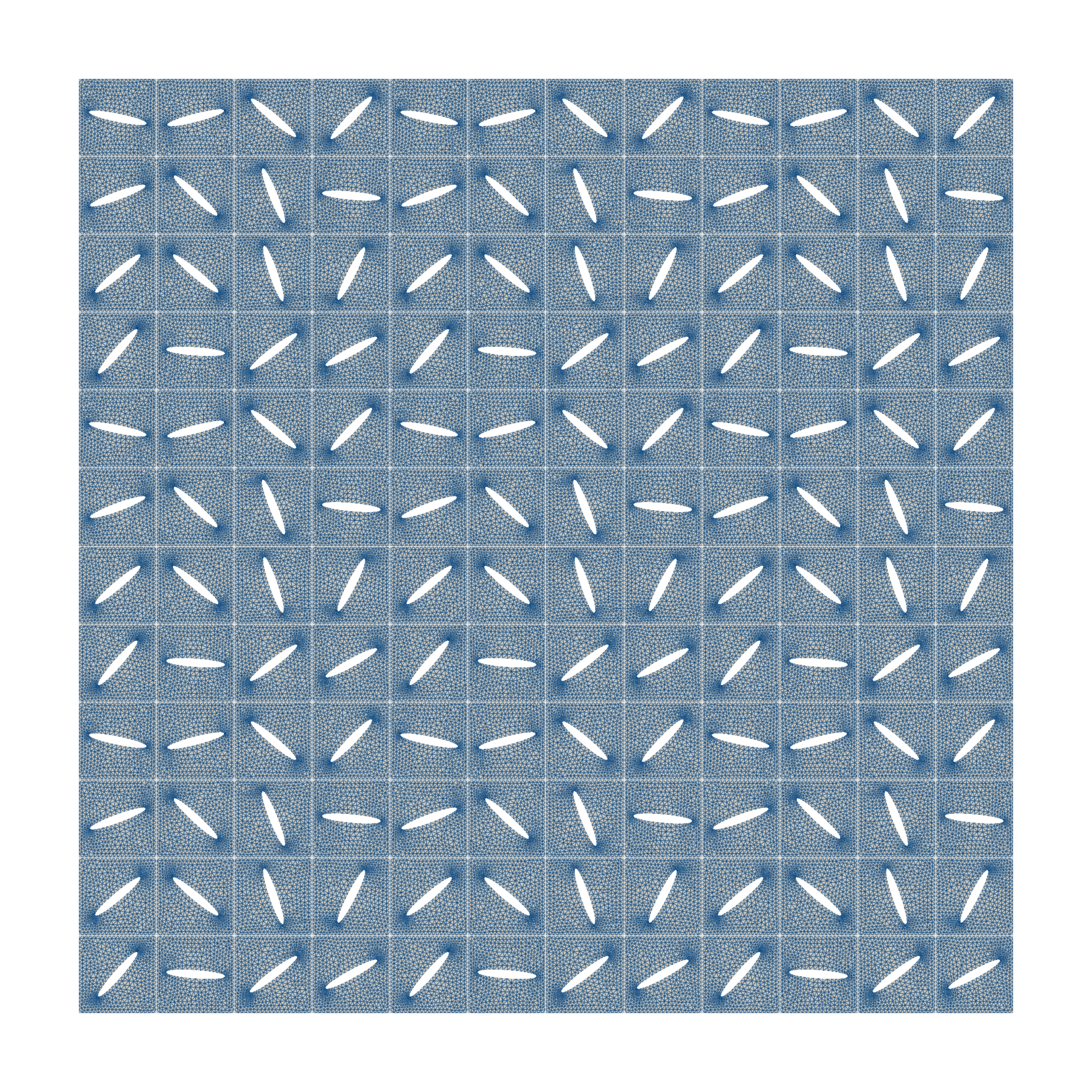}
\caption{\label{fig:comp_model} 
Skin grafting model with a random configuration of cuts. The 
finite element mesh consists of roughly 165,000 finite elements.}
\end{figure}

\begin{figure}[hbt]
	\centering %
	\begin{minipage}{0.4\linewidth} 
		\centering
		\includegraphics[width=\linewidth]{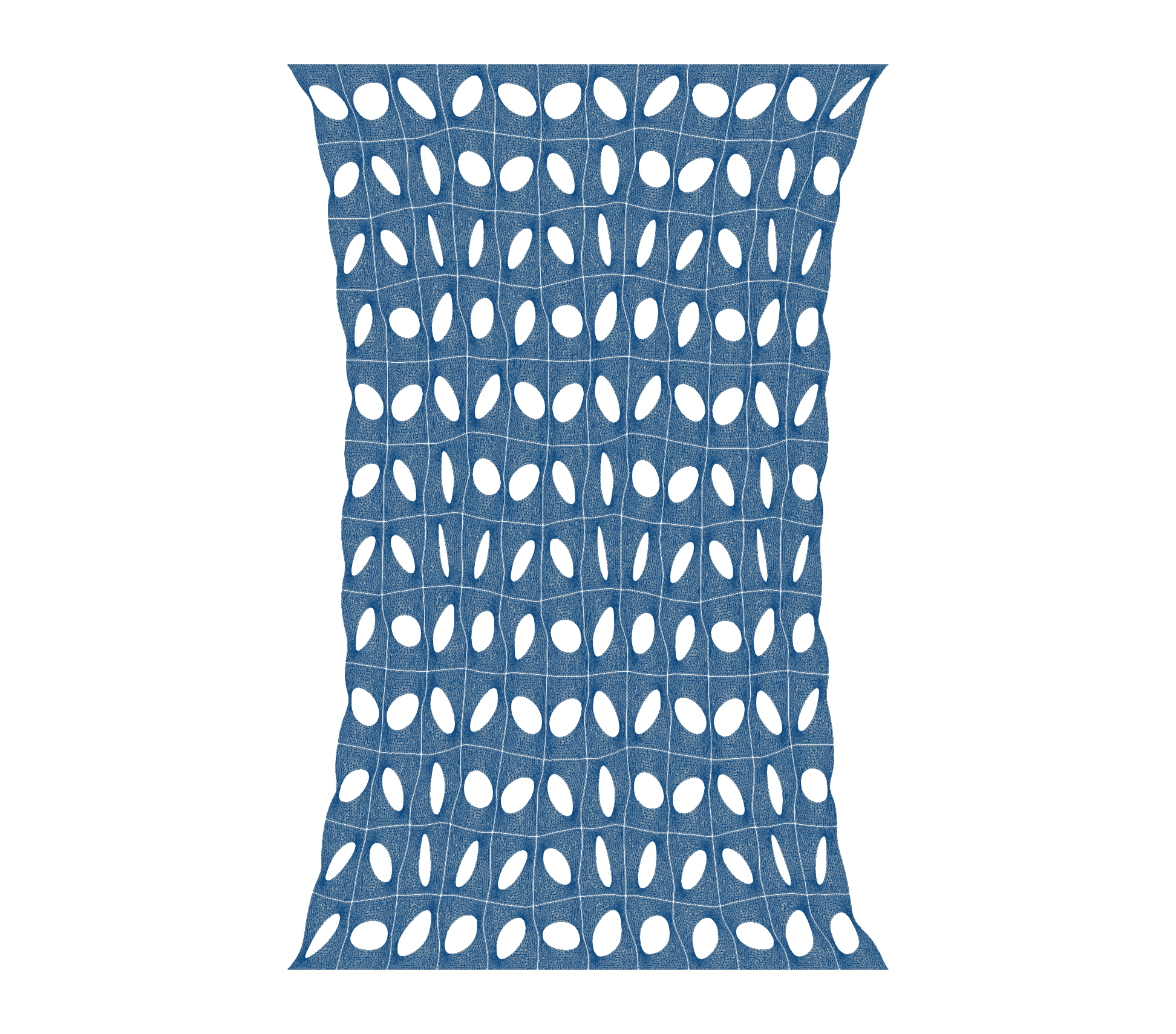} (b)
	\end{minipage}
	\begin{minipage}{0.4\linewidth} 
		\centering
		\includegraphics[width=\linewidth]{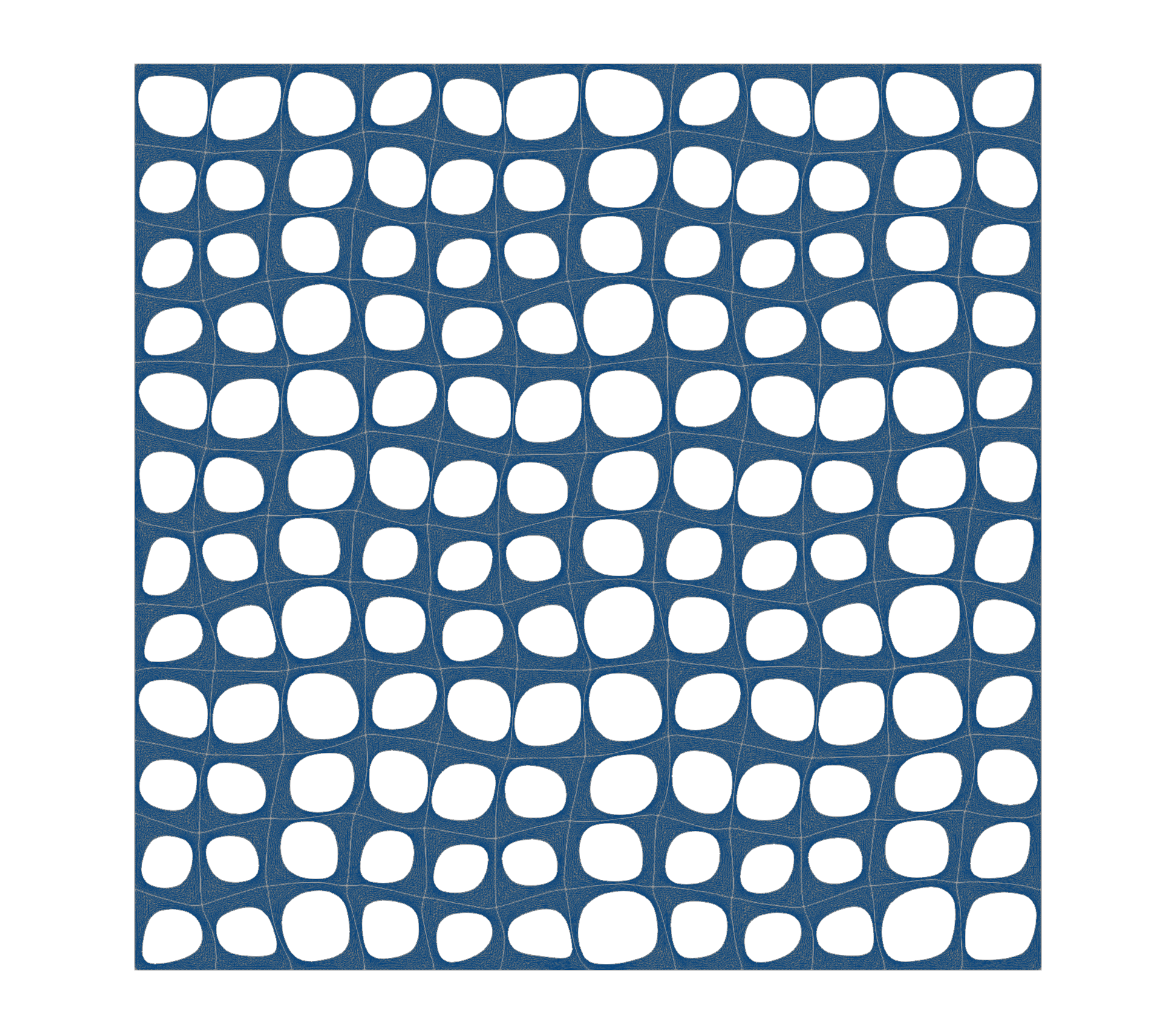} (c)
	\end{minipage}
	\caption{\label{fig:stretched} 
	Stretched skin for the random configuration of cuts: 
	(a) single axis stretching, (b) bi-axial stretching.
	} 
\end{figure}

For solving the equations of linear elasticity on the respective
layout, we apply the finite element method as provided by
the finite element solver FreeFem++, see \cite{hecht2012new}. 
The material parameters are set as proposed in the review article 
\cite{kalra2016mechanical}: the Young modulus is \(E=50\) MPa 
and the Poisson is ratio \(\nu=0.48\). Since the finite element mesh 
is automatically adapted towards the current geometry during the 
optimization process, it is sufficient to say that the finite element 
mesh size  varies from \(h_\text{min} = 6.5\cdot10^{-4}\) to 
\(h_\text{min} = 7.5\cdot10^{-3}\), which leads to a number of about 
165,000 finite elements. An illustration of the finite element 
mesh in case of a random configuration can be found in 
Figure~\ref{fig:comp_model}.

We consider two stress models. The first one is stretching in one 
axial direction, i.e.,\(\GD\) are two opposite boundaries of the square 
\(\Omega\), on which there is a Dirichlet condition, and the remaining 
boundaries \(\GN\) are free. The second one is the stretching in both 
axial directions, i.e., a Dirichlet boundary condition is prescribed at 
the whole boundary \(\partial\Omega = \GD\). In both cases, we set 
\(\bg=0.25\bn\) and \(\bf=\boldsymbol{0}\). Note that, with this setup, 
the solution of the equations of linear elasticity will admit singularities 
in the vertices of the square. A visualization of the stretched skin in
case of the random configuration from Figure \ref{fig:comp_model} 
can be found in Figure \ref{fig:stretched}.

\subsection{Numerical results: Stretching in one axial direction}
In the case of stretching in one axial direction, the gradient 
descent method provides satisfactory results. The initial 
configuration is the same for all experiments and the one 
presented in Figure~\ref{fig:comp_model}. A total of 100 
iterations is performed, and the outcomes obtained are 
presented in Figure \ref{fig:res_1ax}. Moreover, in all our 
subsequent experiments, we have chosen \(p=5\) when 
minimizing the \(L^p\)-norm of the von Mises stress.

When minimizing the compliance \(\C(\Oa)\), the optimal
design is a horizontal arrangement of the cuts. More 
precisely, all cuts are perpendicular to the stretch direction, 
which is a reasonable configuration, see the first row in the 
Figure \ref{fig:res_1ax}. It allows for the largest possible 
opening of the cuts, which makes the stretching process 
easier. The deformation in this situation is characterized 
by a minimum amount of work. However, the skin in the 
area between the cuts is quite thin, which increases its
stress. Therefore, it is not surprising that in order to 
decrease the stress intensity, it is necessary to find an 
better balance between stretching easiness and maintaining 
sufficient thickness in the areas between the cuts. Thus, 
when minimizating the \(L^5\)-norm of the von Mises stress 
\(\M(\Oa)\), the optimal design is a zigzag configuration,
compare the second row in the Figure~\ref{fig:res_1ax}.

\begin{figure}[hbt]
	\centering %
	\begin{minipage}{0.32\linewidth} 
		\centering
		\includegraphics[width=\linewidth]{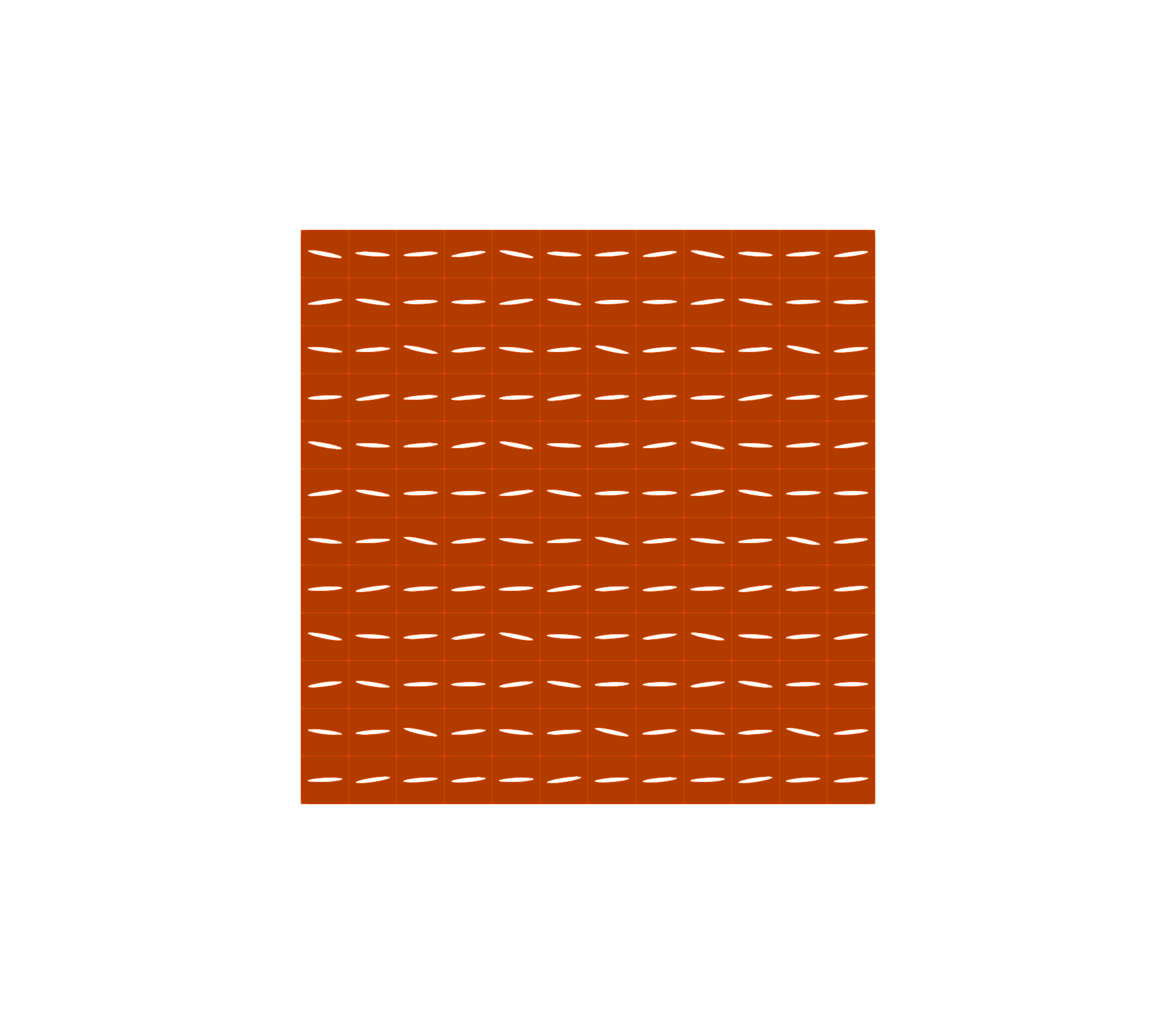} 
	\end{minipage}
	\hfill 
	\begin{minipage}{0.32\linewidth} 
		\centering
		\includegraphics[width=\linewidth]{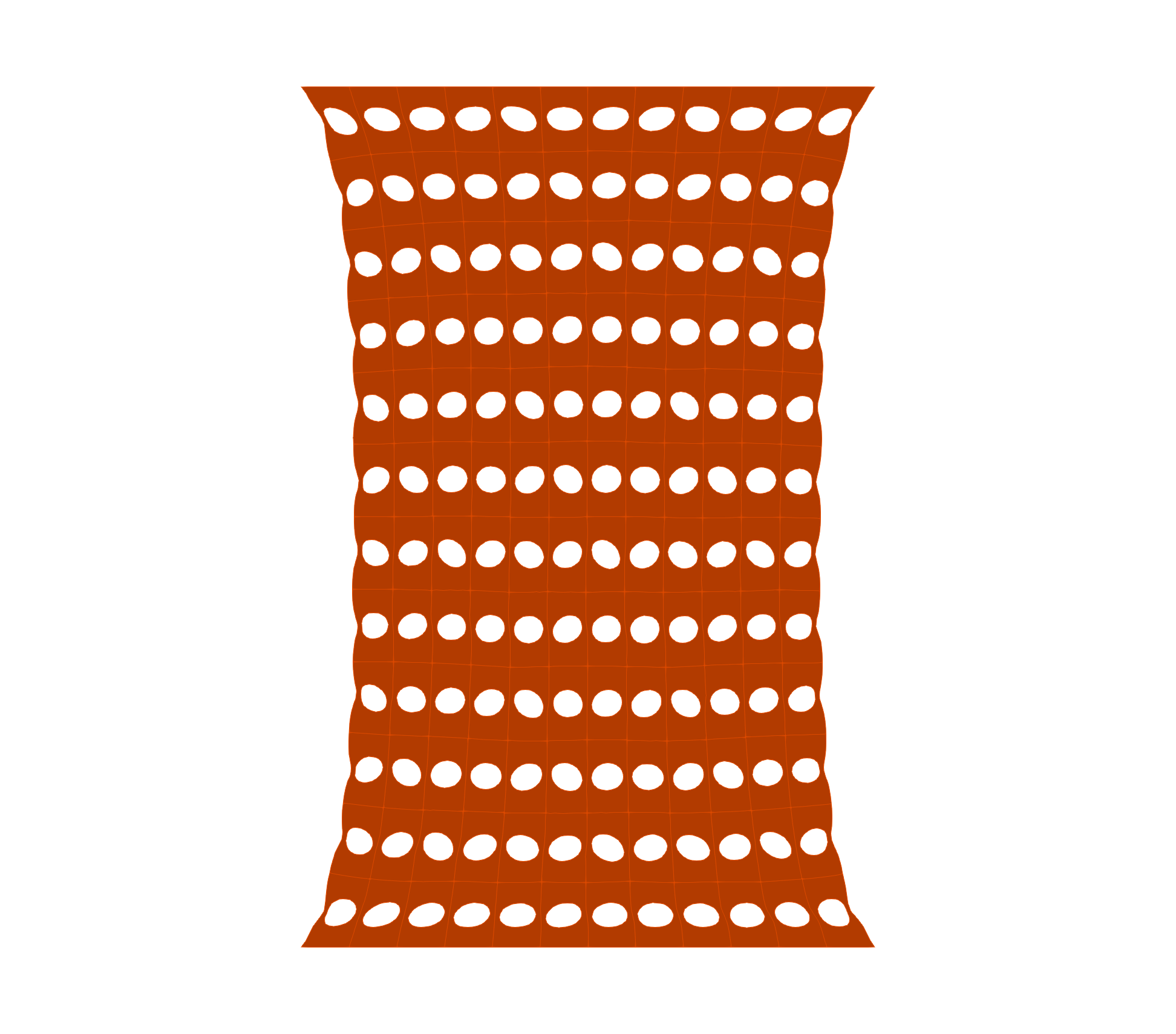} 
	\end{minipage}
	\hfill 
	\begin{minipage}{0.33\linewidth} 
		\centering
		\includegraphics[width=\linewidth]{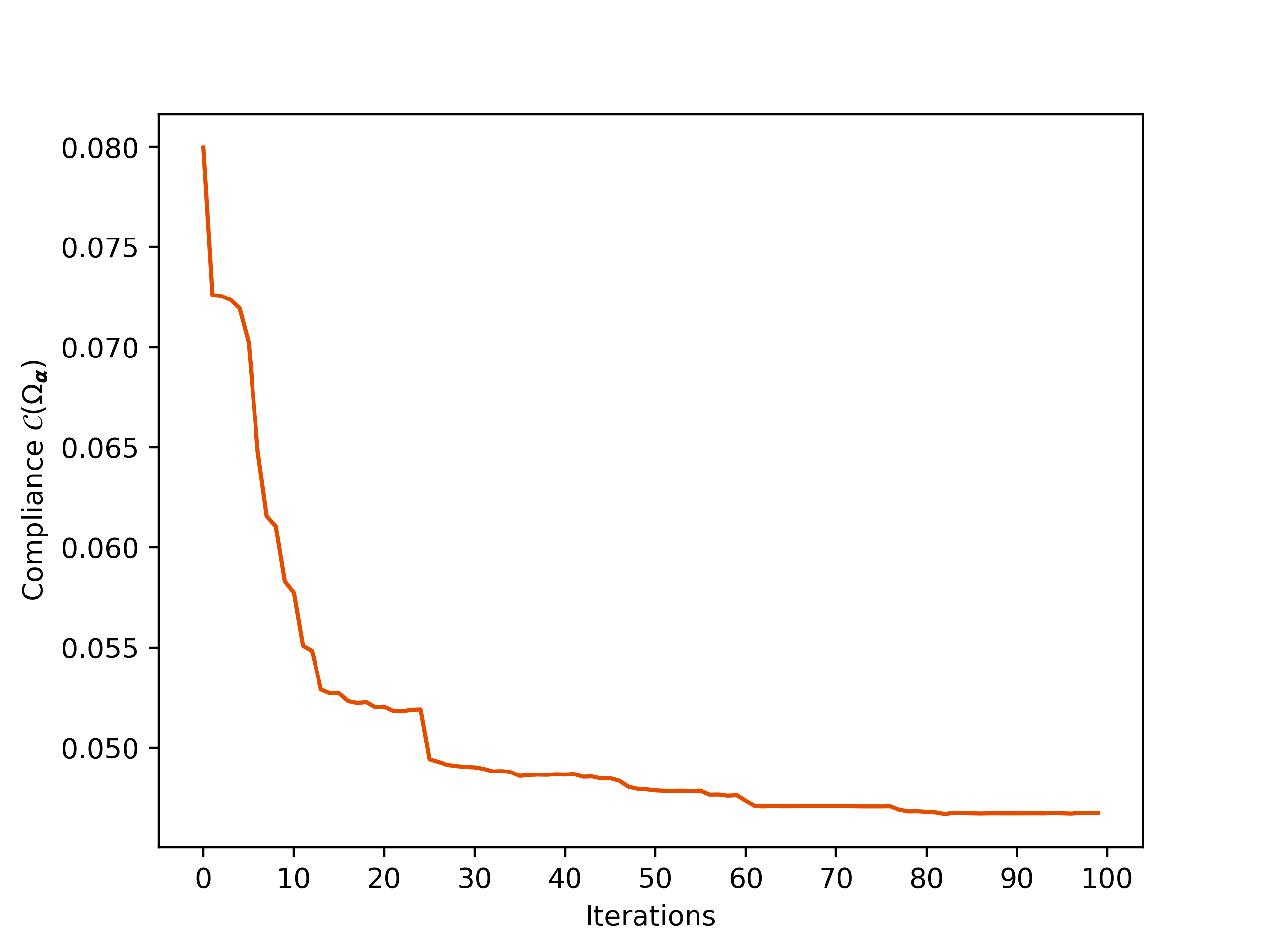} 
	\end{minipage}
	\vfill
	\begin{minipage}{0.32\linewidth} 
	\centering
	\includegraphics[width=\linewidth]{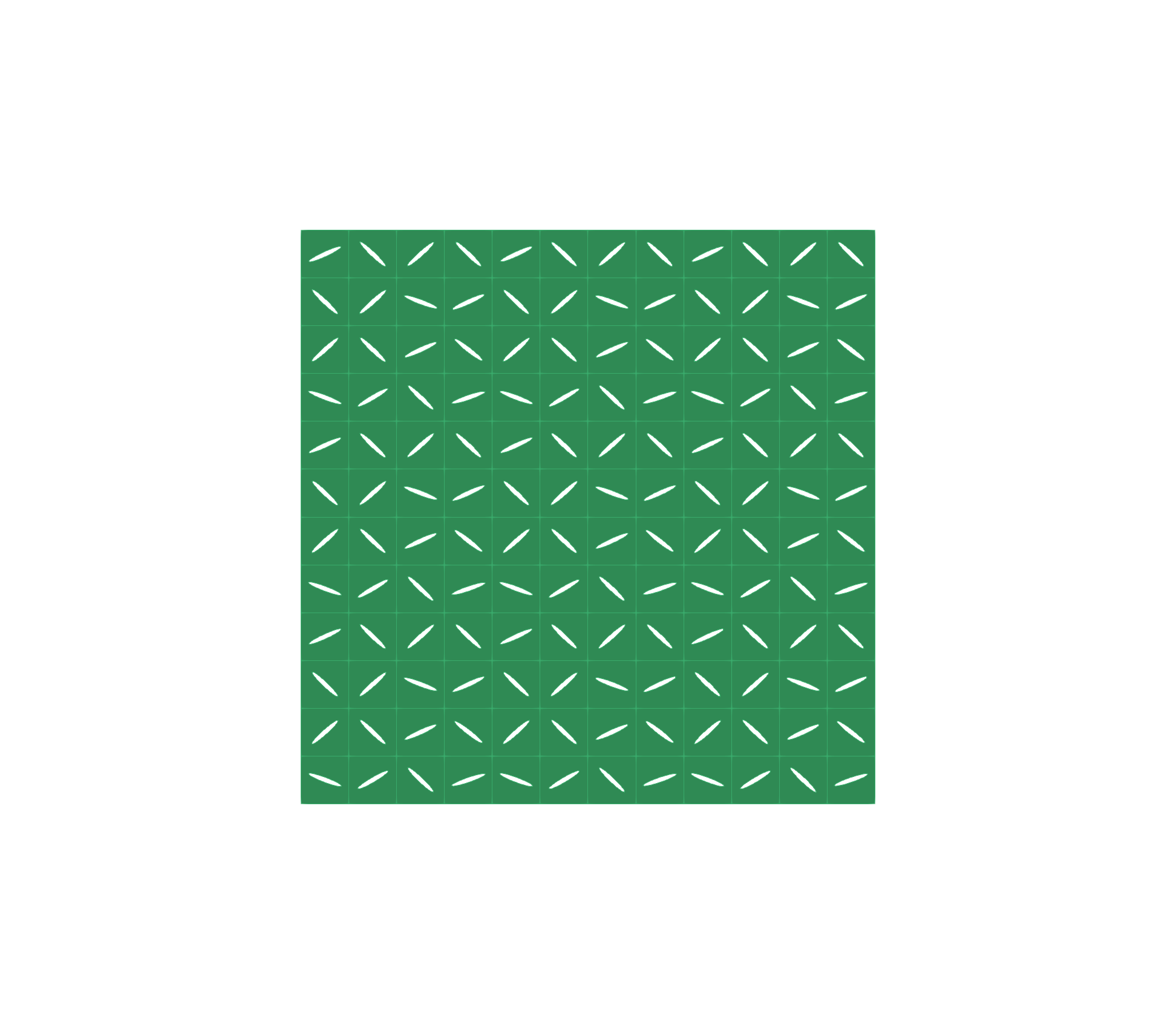} 
	\end{minipage}
	\hfill 
	\begin{minipage}{0.32\linewidth} 
		\centering
		\includegraphics[width=\linewidth]{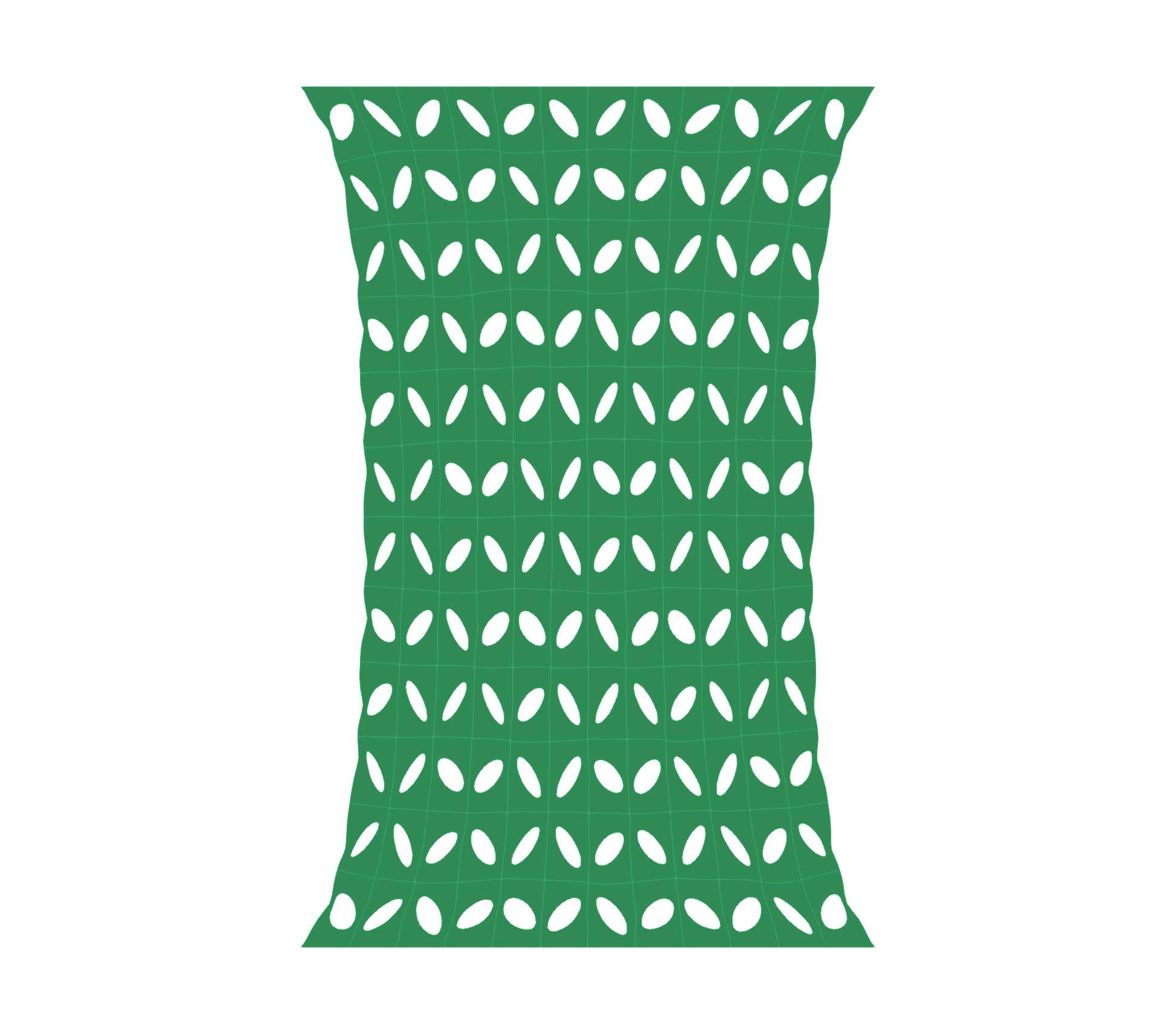} 
	\end{minipage}
	\hfill 
	\begin{minipage}{0.33\linewidth} 
		\centering
		\includegraphics[width=\linewidth]{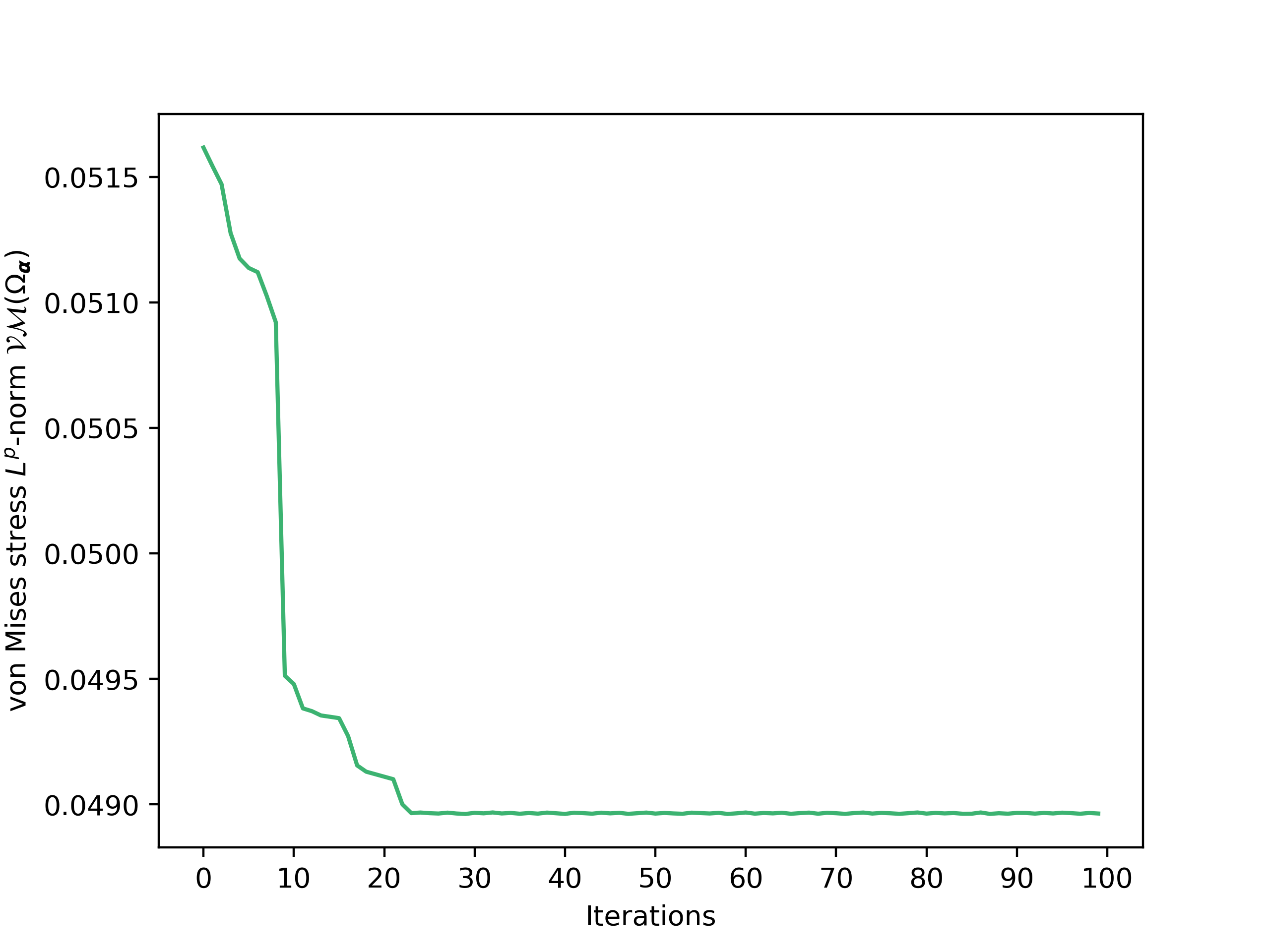} 
	\end{minipage}
		\vfill
	\begin{minipage}{0.32\linewidth} 
		\centering
		\includegraphics[width=\linewidth]{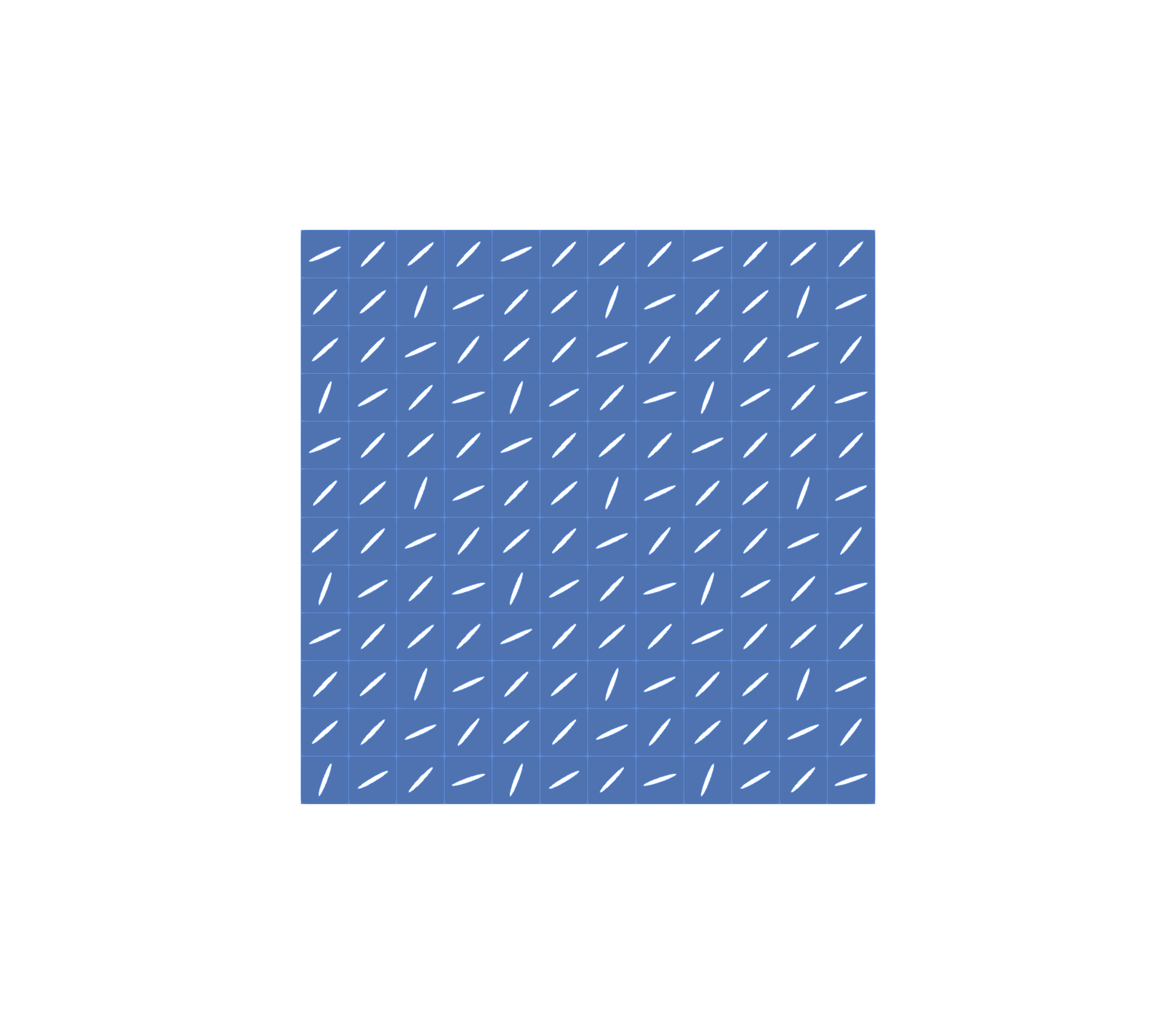} 
	\end{minipage}
	\hfill 
	\begin{minipage}{0.32\linewidth} 
		\centering
		\includegraphics[width=\linewidth]{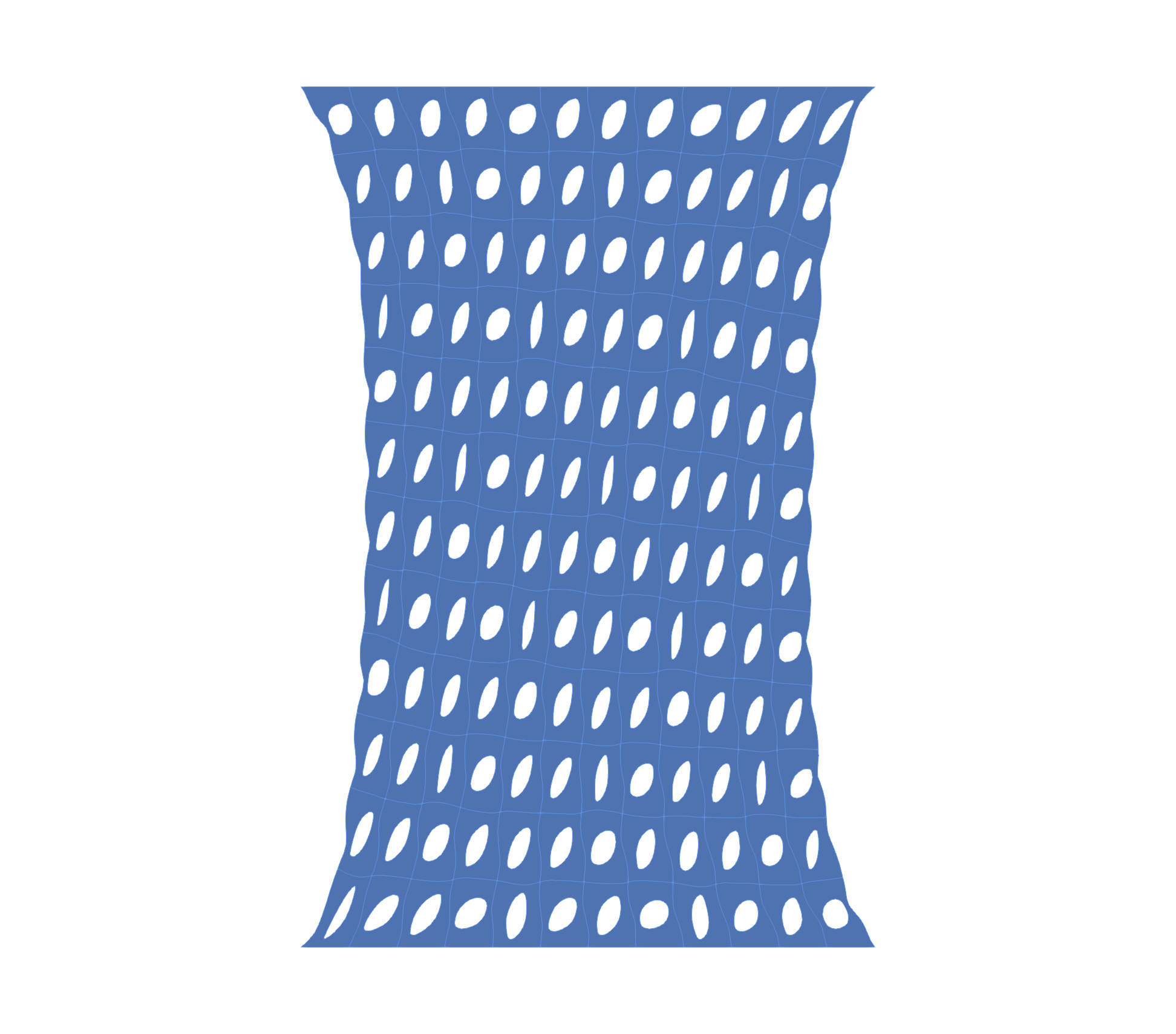} 
	\end{minipage}
	\hfill 
	\begin{minipage}{0.33\linewidth} 
		\centering
		\includegraphics[width=\linewidth]{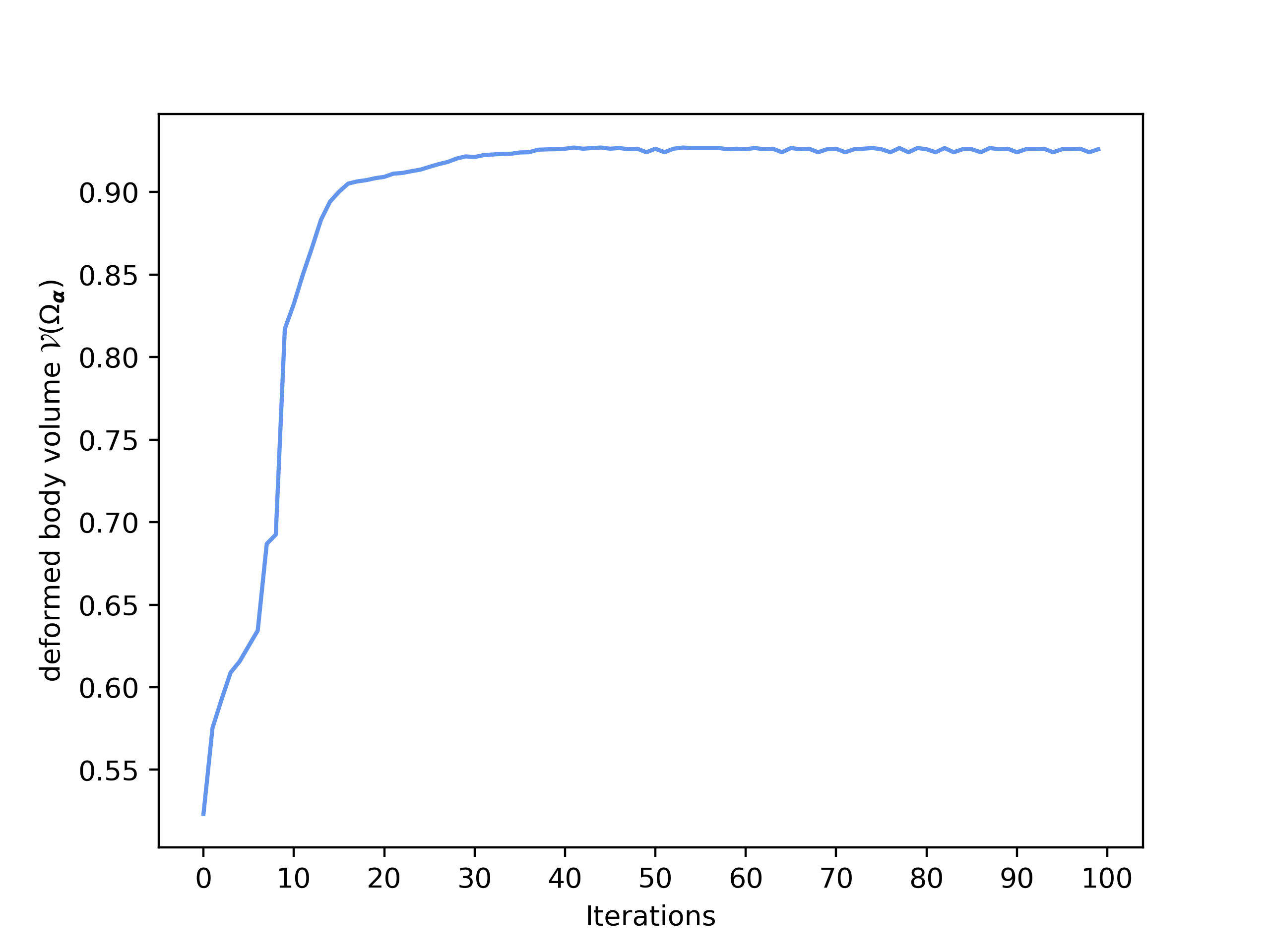} 
	\end{minipage}
	\caption{
		Stretching in one axial direction. Final cut configuration 
		(left), deformed skin (middle), and convergence history (right): 
		first row -- compliance \(\C(\Oa)\), 
		second row -- \(L^5\)-norm of the von Mises stress \(\M(\Oa)\), 
		third row -- area of the deformed body \(\A(\Oa)\).
	} 
	\label{fig:res_1ax} 
\end{figure}

\begin{figure}[hbt]
	\centering %
	\begin{minipage}{0.32\linewidth} 
		\centering
		\includegraphics[width=\linewidth]{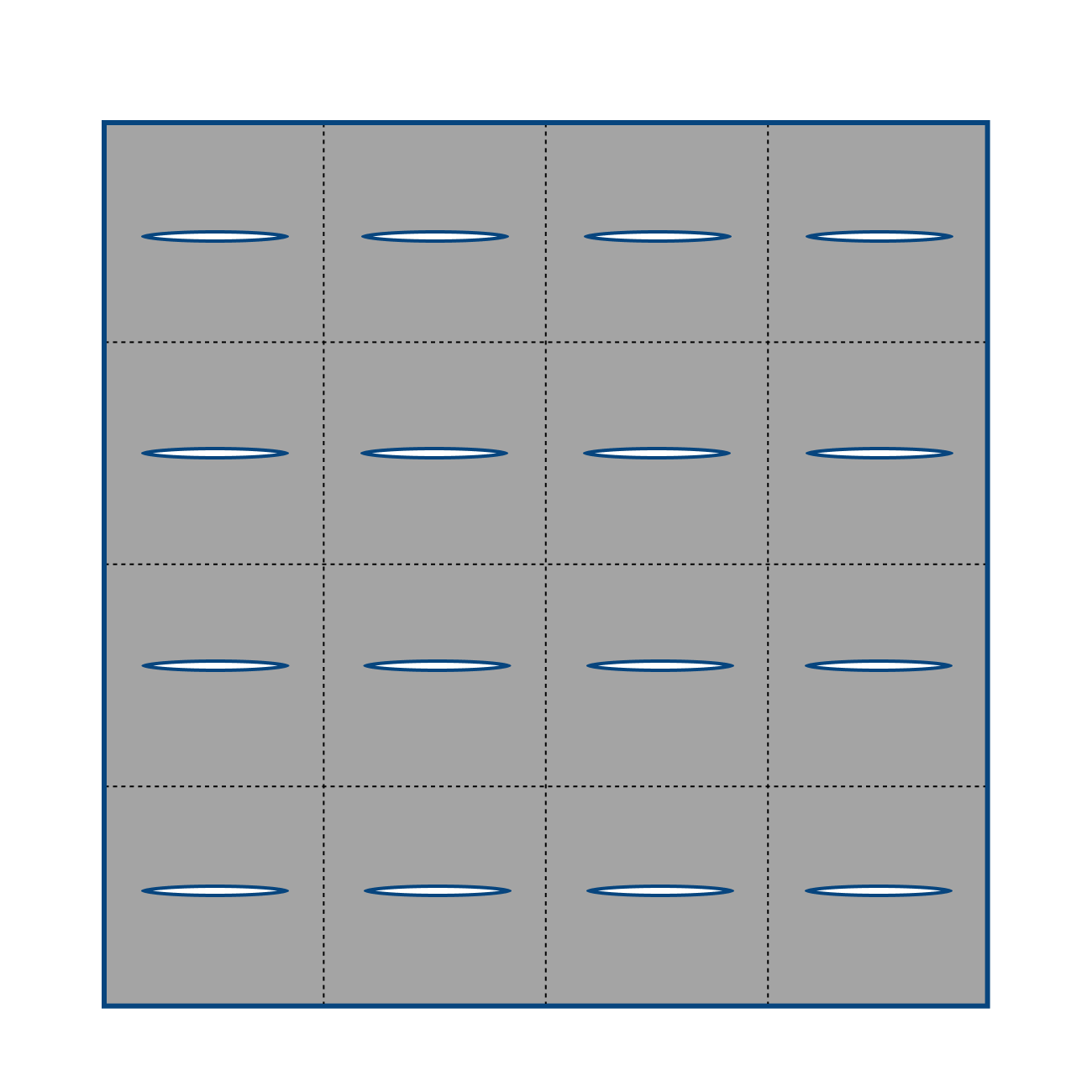} Horizontal lines
	\end{minipage}
	\hfill 
	\begin{minipage}{0.32\linewidth} 
		\centering
		\includegraphics[width=\linewidth]{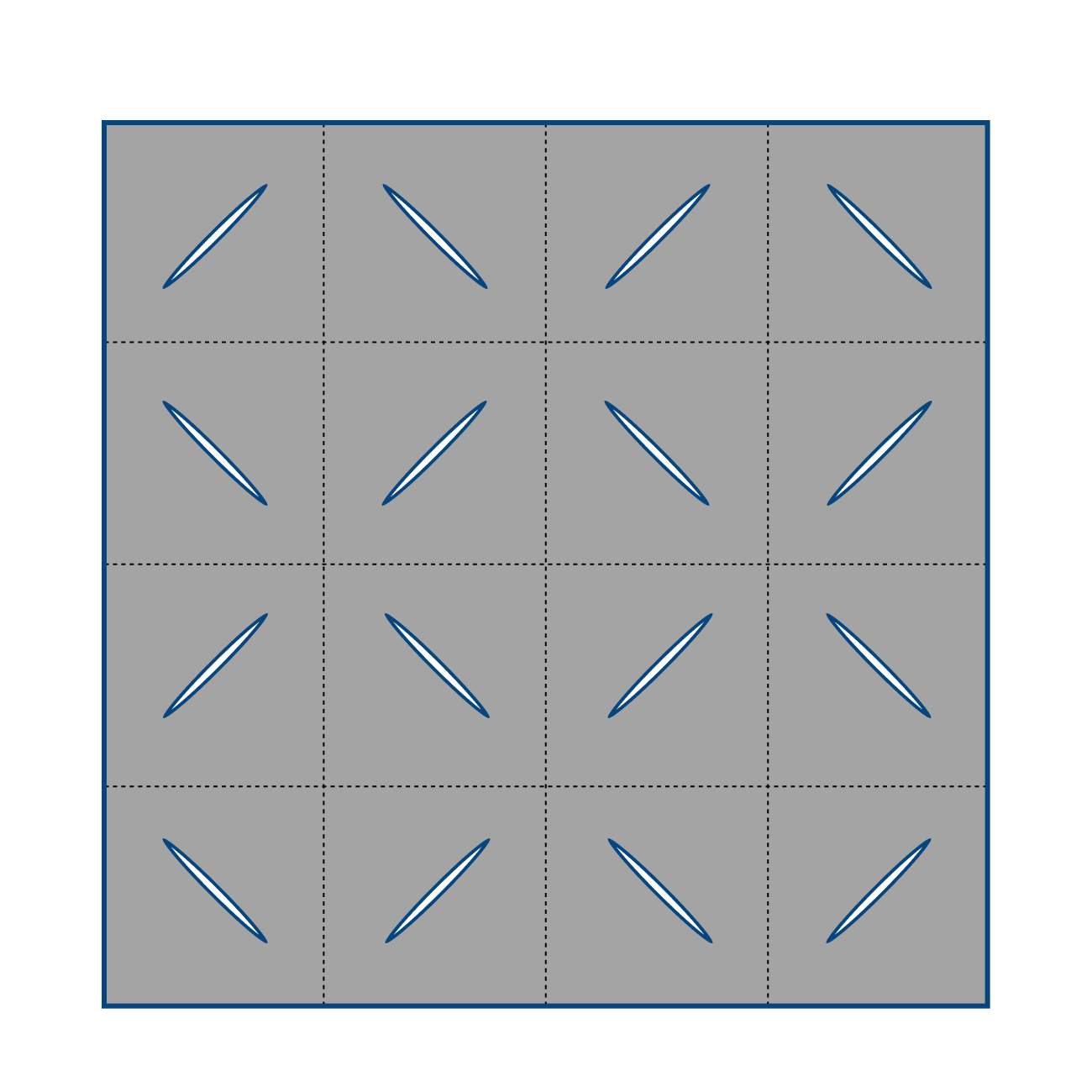} Zigzag I
	\end{minipage}
	\hfill 
	\begin{minipage}{0.32\linewidth} 
		\centering
		\includegraphics[width=\linewidth]{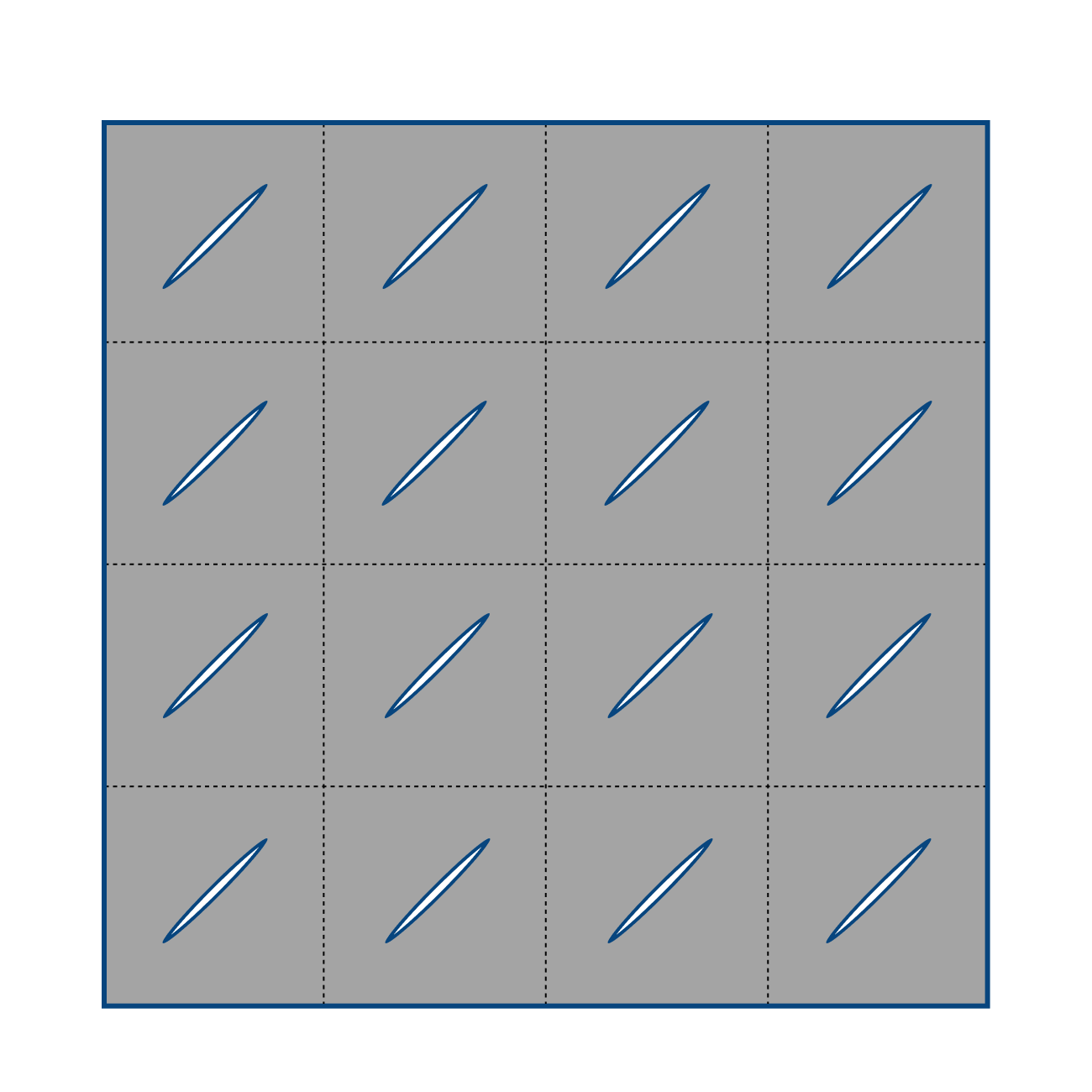} Diagonal lines
	\end{minipage}
	\caption{\label{fig:configsI}Illustration of the optimal designs
	in case of stretching the skin in one axial direction.}
\end{figure}

Finally, in the case of the maximization of the area of 
the stretched skin, i.e., when considering the shape functional 
\(\A(\Oa)\), we on the one hand maximize the area of the 
entire deformed body \(\A(\Omega_{\bua})\) and on the other
hand minimize the size of the deformed cuts \(\A((\oa)_{\bua})\). 
Therefore, the logical result is a diagonal configuration of the 
cut, see the third row in the Figure \ref{fig:res_1ax}. Here, 
the contraction of the body perpendicular to stretching 
and the deformation of the cuts are in balance. 

The optimal configurations found by the optimization process 
are summarized in Figure~\ref{fig:configsI} without computational 
noise in the outcomes caused by the simulation process.

\subsection{Numerical results: Stretching in two axial directions}
In the case of stretching the skin in both axial directions, 
solely using the gradient descent method is not successful
since we get stuck in one of the many local minima. However,
the combination of the gradient descent method with the 
genetic algorithm gives reasonable results. The initial 
population consisted of 150 random individuals as 
specified in Figure \ref{fig:res_1ax}. It took about 30 iterations 
for each functional to obtain a single elite individual. The results 
can be found in Figure \ref{fig:res_2ax}. The convergence 
histories represent the values of the objective function for 
the fittest individual in the population during the 
algorithm’s execution.

In the case of compliance minimization \(\C(\Oa)\), we obtained a 
configuration that switched between horizontal and vertical cuts, 
see the first row in the Figure \ref{fig:res_2ax}. This configuration 
is remarkable because it gives the skin auxetic properties. Unlike 
conventional structures, auxetic ones expand perpendicular to the 
stretching axes rather than shrinking. This results in a negative 
Poisson's ratio. With this configuration, the work required to deform 
the body is the minimized. We get also the same configuration when 
maximizing the area of the deformed body \(\A(\Oa)\).\footnote{By
taking a closer look to Figure~\ref{fig:res_2ax}, one readily infers
that the swap pattern is shifted, compare also Figure~\ref{fig:configsII}.
But for symmetry reasons, both patterns represent minimizers with
identical values in the respective objective functional.}

\begin{figure}[hbt]
	\centering %
	\begin{minipage}{0.32\linewidth} 
		\centering
		\includegraphics[width=\linewidth]{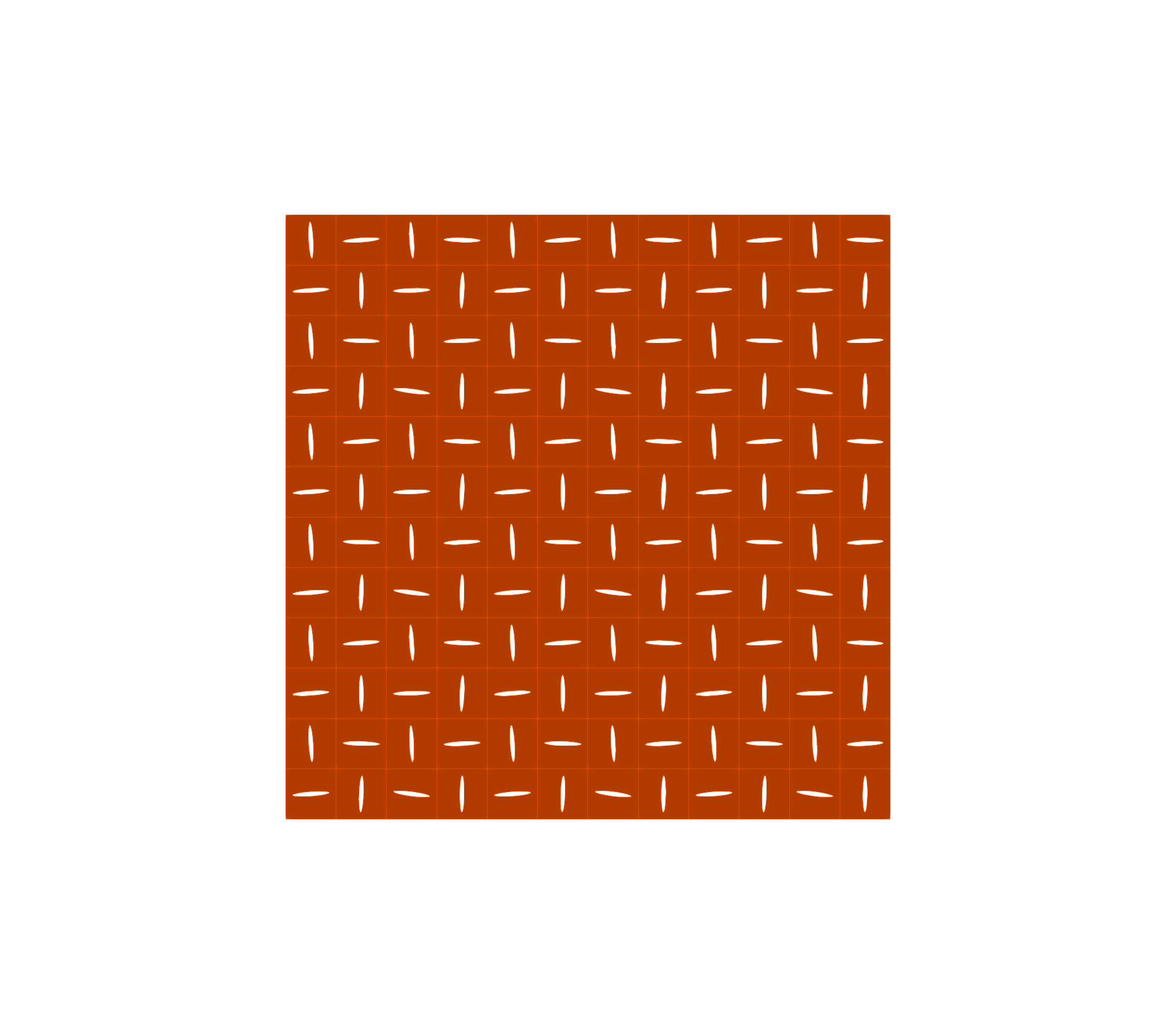} 
	\end{minipage}
	\hfill 
	\begin{minipage}{0.32\linewidth} 
		\centering
		\includegraphics[width=\linewidth]{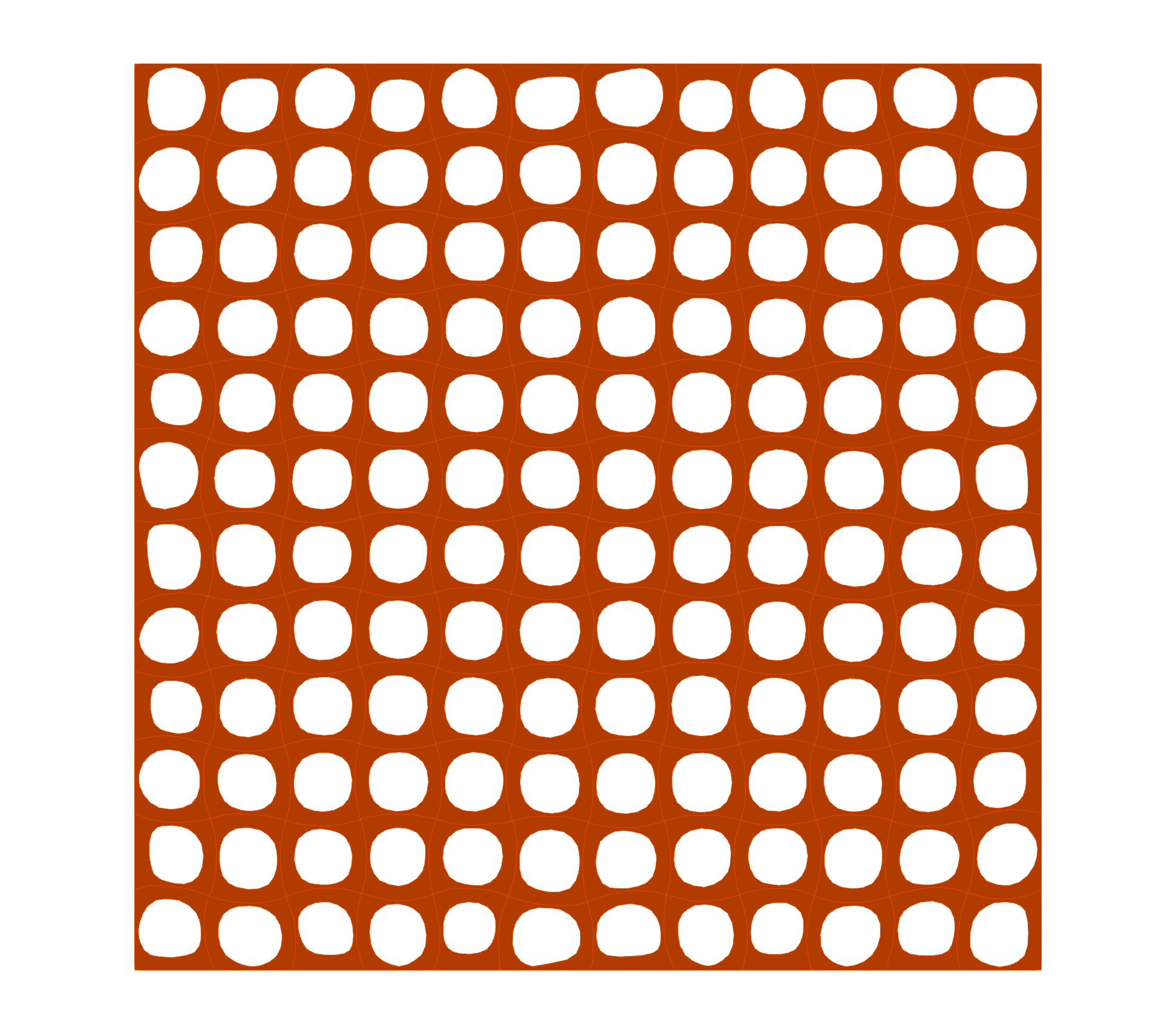} 
	\end{minipage}
	\hfill 
	\begin{minipage}{0.33\linewidth} 
		\centering
		\includegraphics[width=\linewidth]{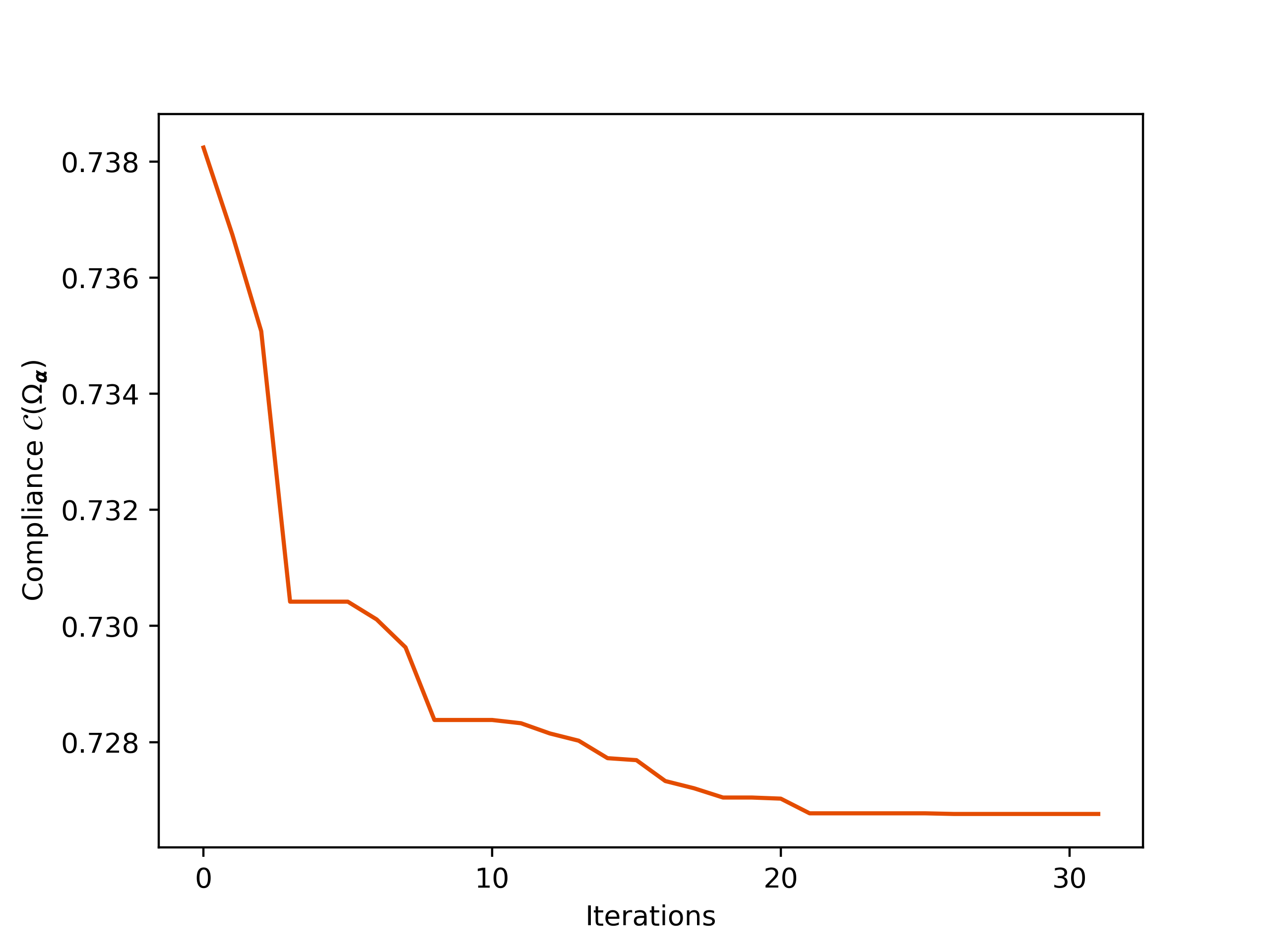} 
	\end{minipage}
	\vfill
	\begin{minipage}{0.32\linewidth} 
		\centering
		\includegraphics[width=\linewidth]{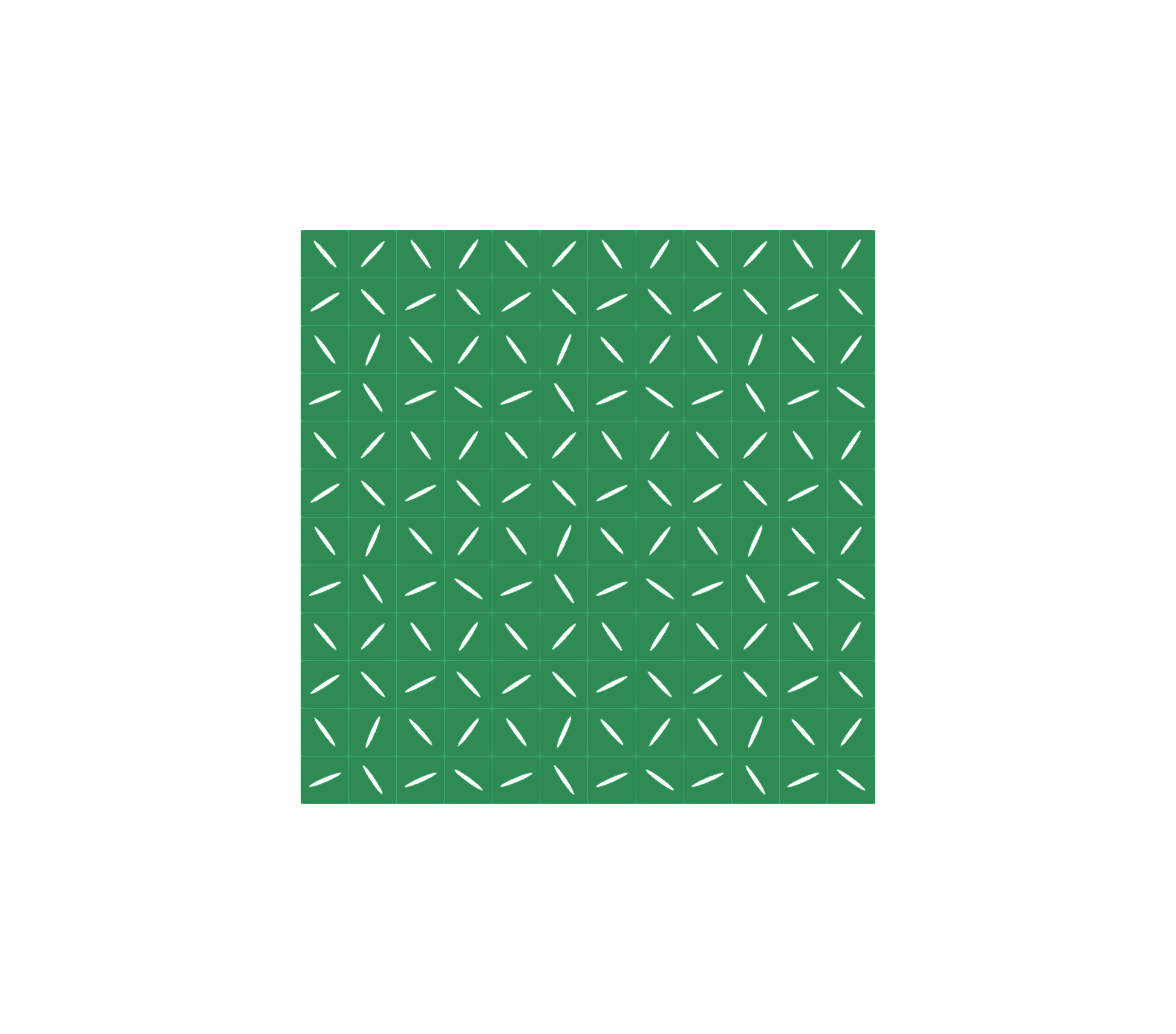} 
	\end{minipage}
	\hfill 
	\begin{minipage}{0.32\linewidth} 
		\centering
		\includegraphics[width=\linewidth]{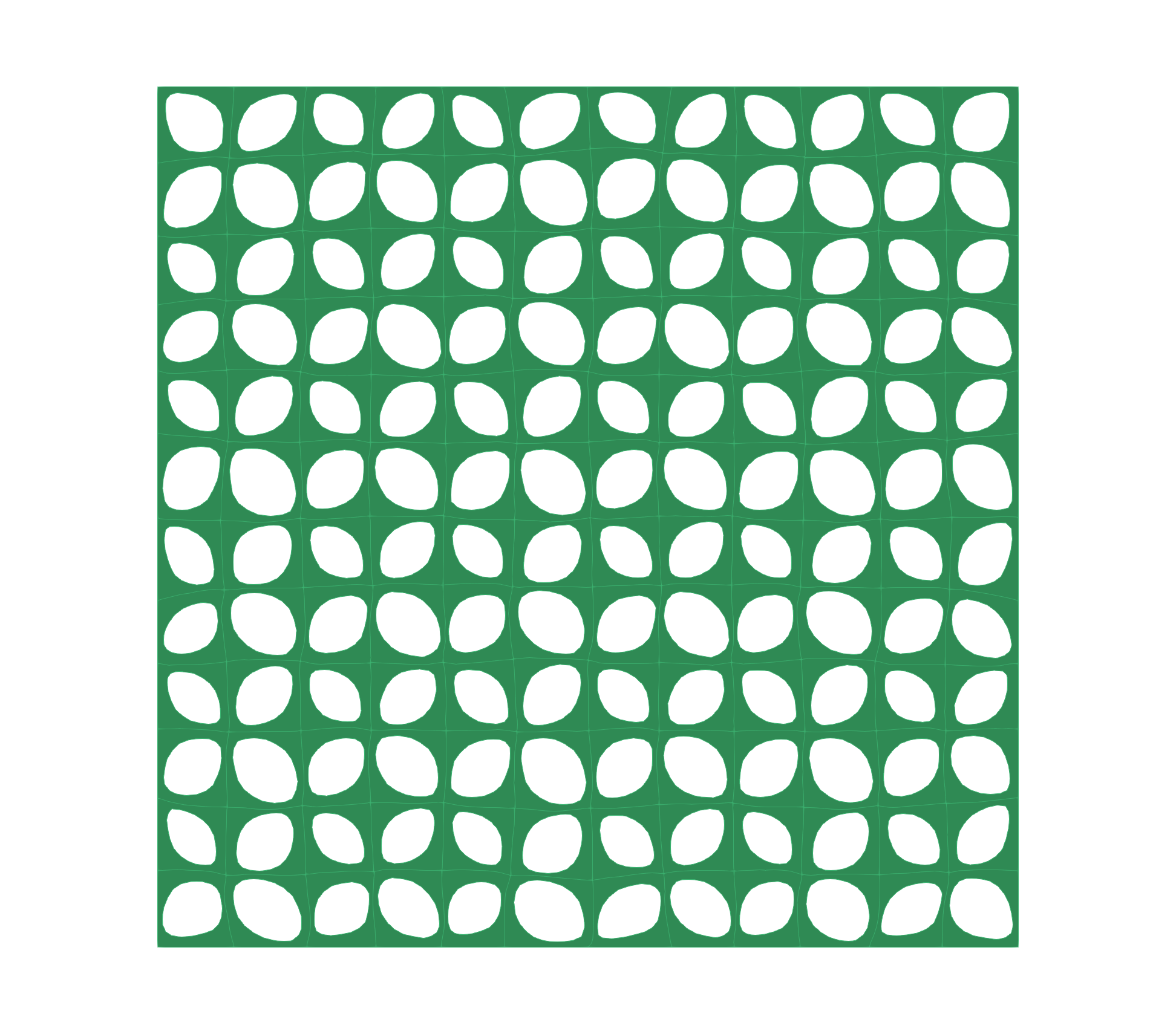} 
	\end{minipage}
	\hfill 
	\begin{minipage}{0.33\linewidth} 
		\centering
		\includegraphics[width=\linewidth]{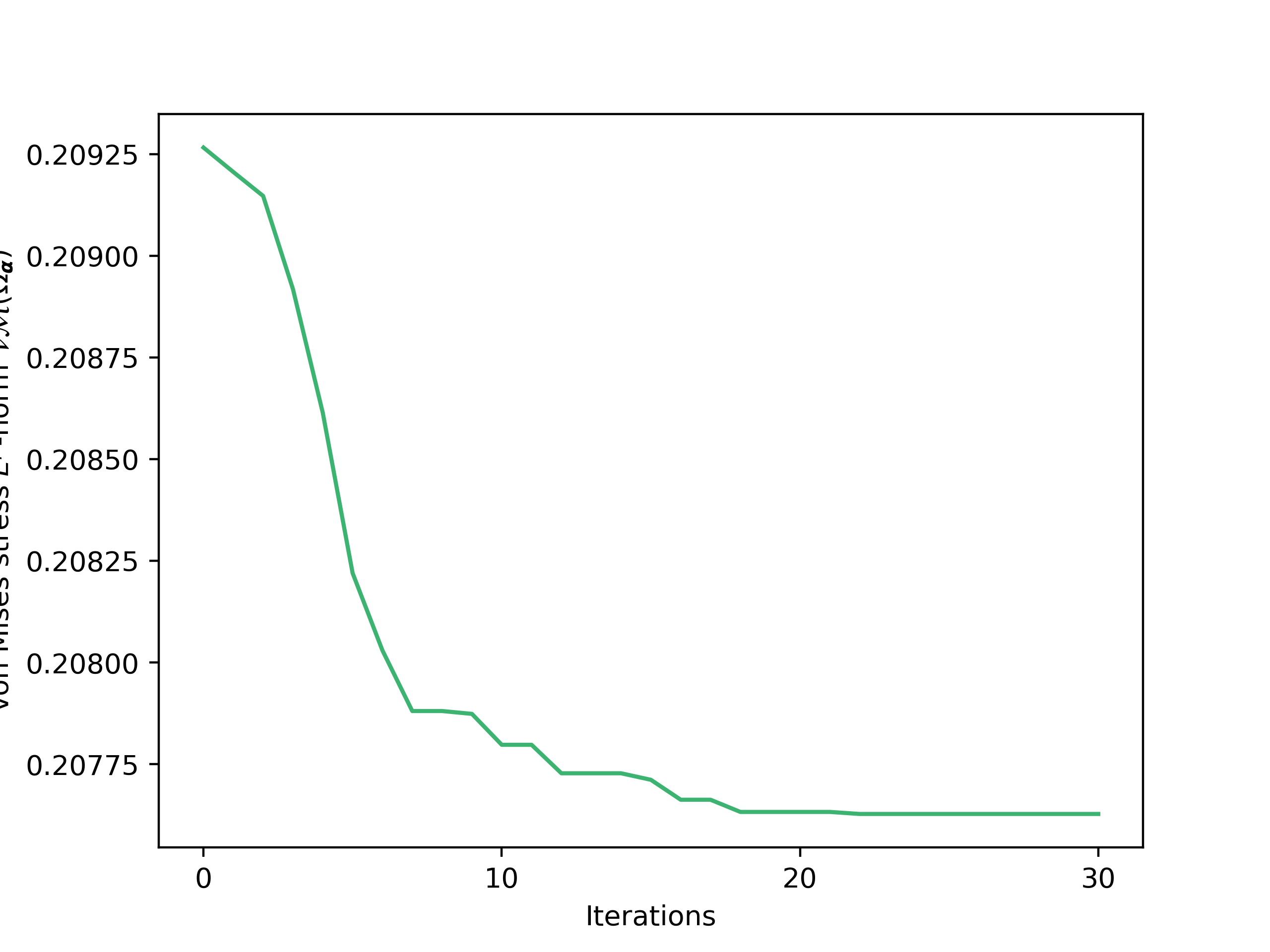} 
	\end{minipage}
	\vfill
	\begin{minipage}{0.32\linewidth} 
		\centering
		\includegraphics[width=\linewidth]{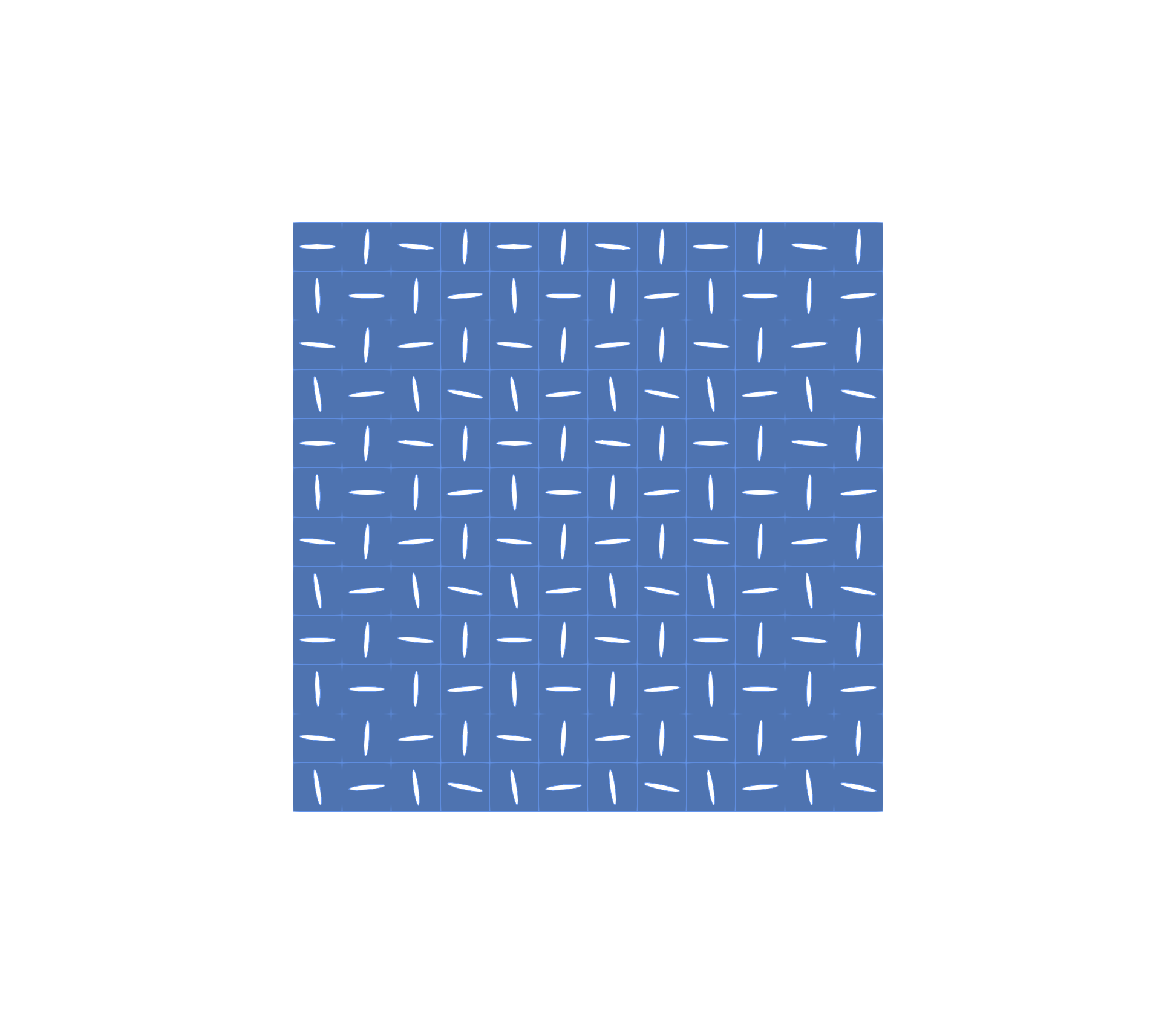} 
	\end{minipage}
	\hfill 
	\begin{minipage}{0.32\linewidth} 
		\centering
		\includegraphics[width=\linewidth]{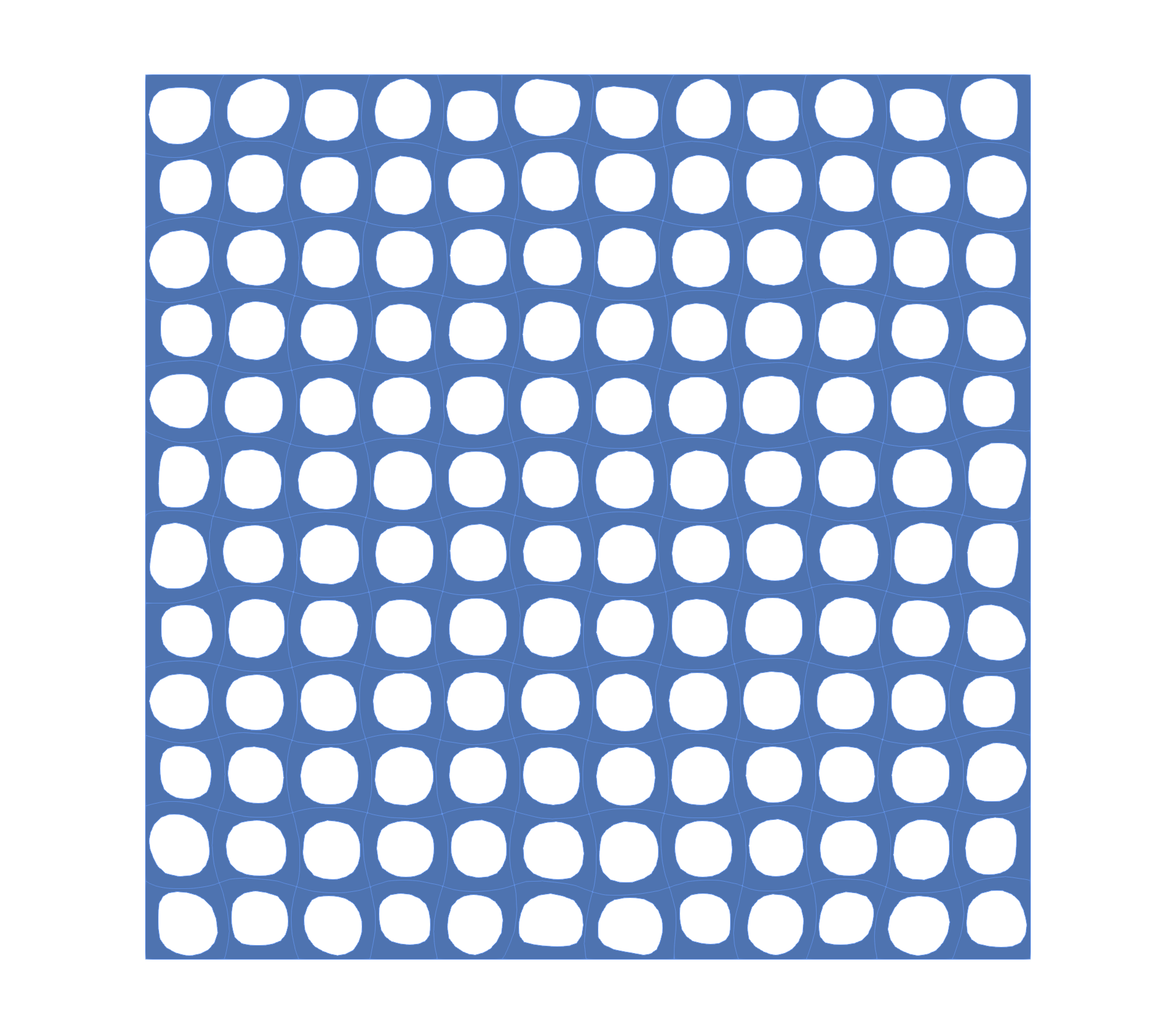} 
	\end{minipage}
	\hfill 
	\begin{minipage}{0.33\linewidth} 
		\centering
		\includegraphics[width=\linewidth]{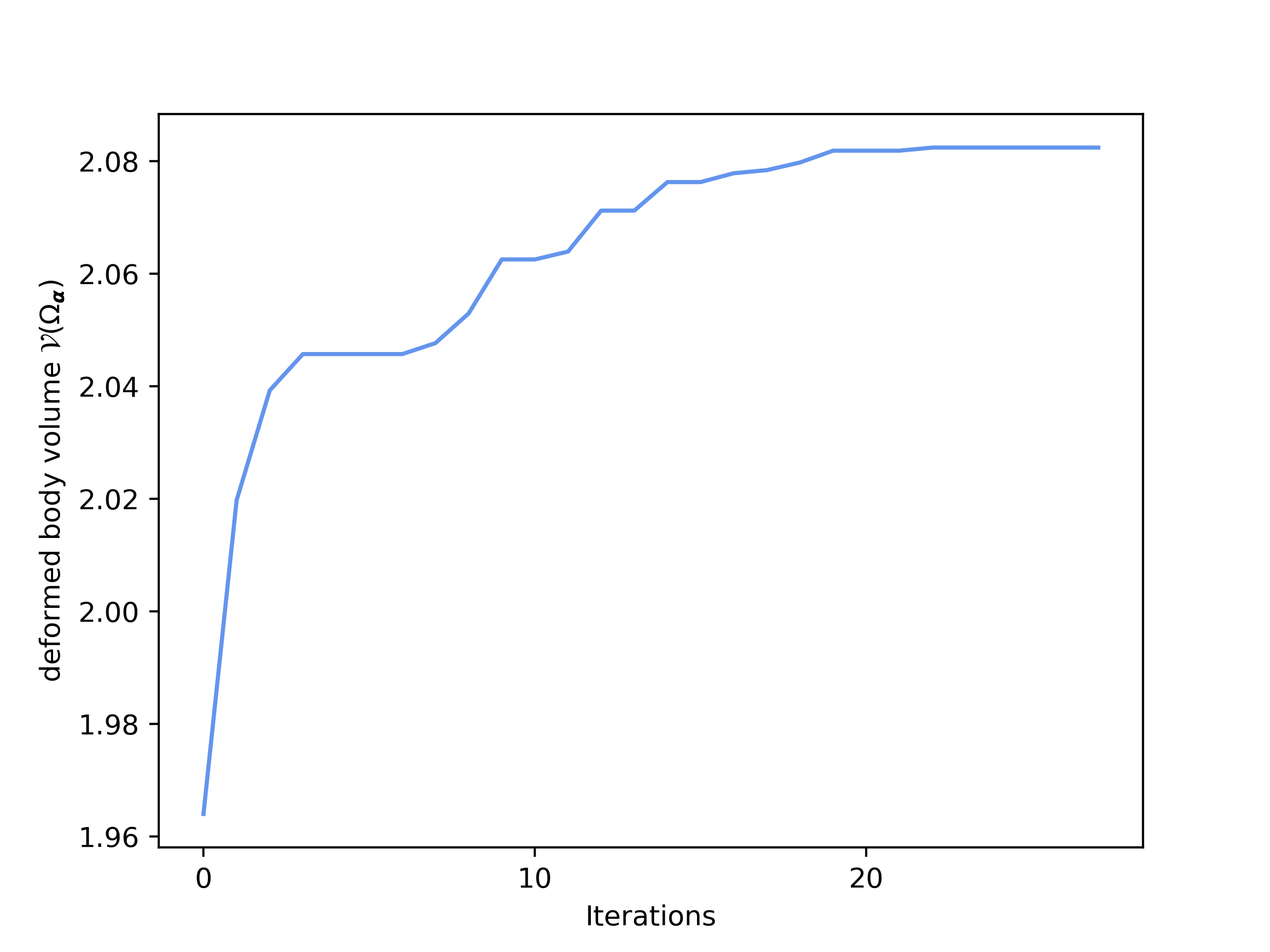} 
	\end{minipage}
	\caption{
		Stretching in two axial directions. Final cut configuration
		(left), deformed skin (middle) and convergence history(right):
		first row -- compliance \(\C(\Oa)\), 
		second row -- \(L^5\)-norm of the von Mises stress \(\M(\Oa)\), 
		third row -- area of the deformed body \(\A(\Oa)\).
	} 
	\label{fig:res_2ax} 
\end{figure}

\begin{figure}[hbt]
	\centering %
	\begin{minipage}{0.32\linewidth} 
		\centering
		\includegraphics[width=\linewidth]{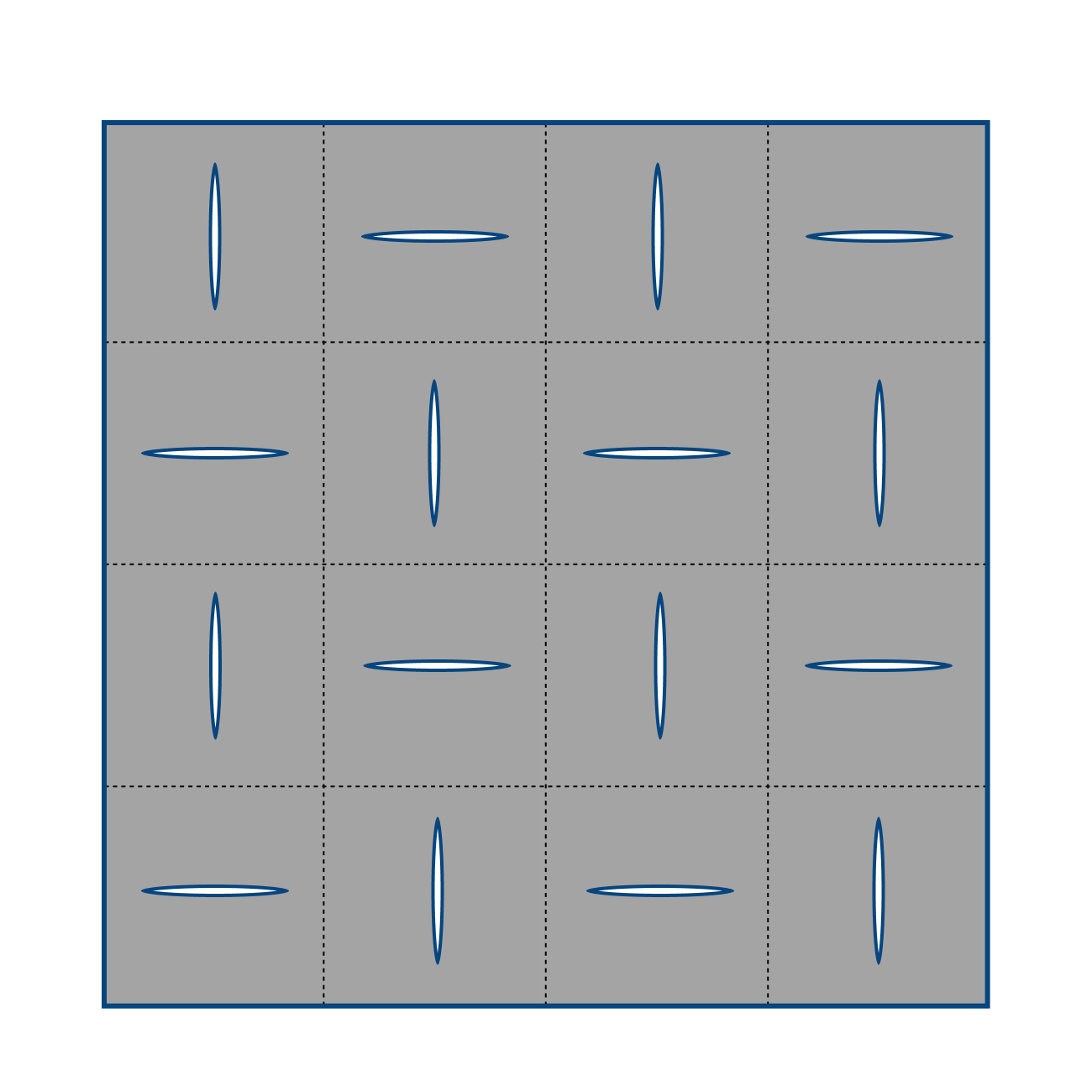} Swap I
	\end{minipage}
	\hfill 
	\begin{minipage}{0.32\linewidth} 
		\centering
		\includegraphics[width=\linewidth]{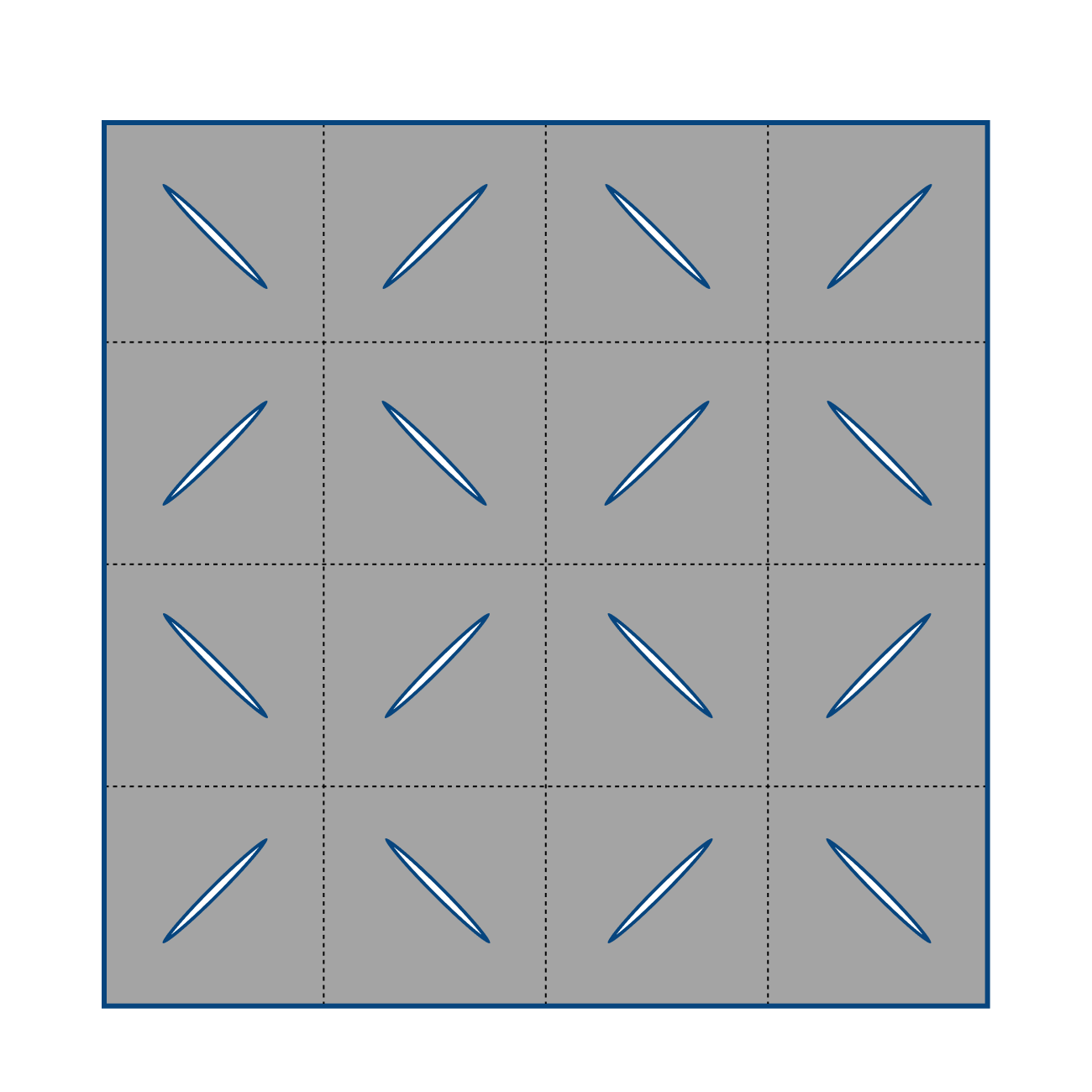} Zigzag I\!I
	\end{minipage}
	\hfill 
	\begin{minipage}{0.32\linewidth} 
		\centering
		\includegraphics[width=\linewidth]{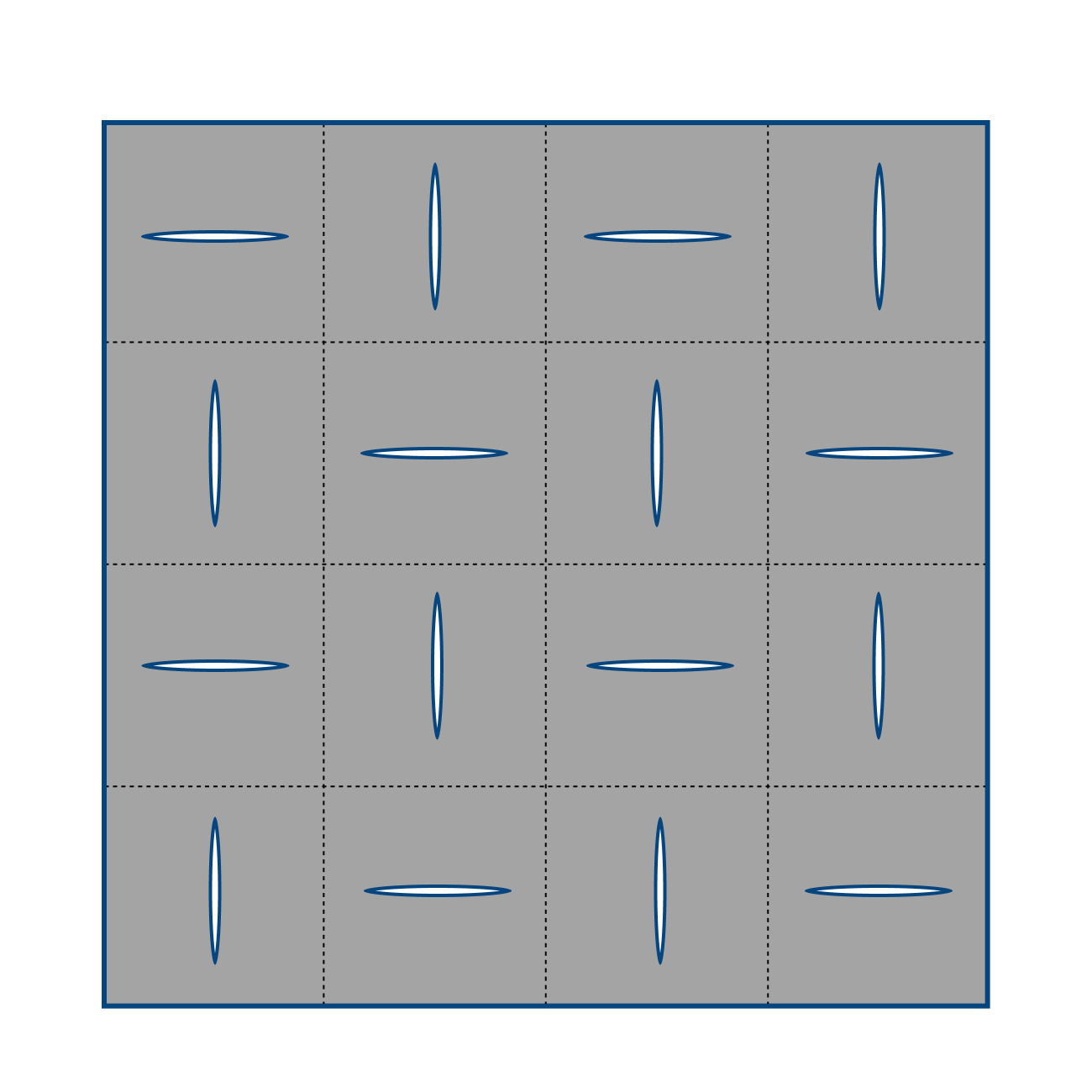} Swap I\!I
	\end{minipage}
	\caption{\label{fig:configsII}Illustration of the optimal designs
	in case of stretching the skin in two axial directions.}
\end{figure}

In the case of stretching in two axial directions, the area of the 
body without taking into account the cuts \(\A(\Omega_{\bua})\) 
is always the same, since we specify the deformation of the entire 
boundary, but the area of the cuts themselves changes. Therefore, 
this problem of maximization of the area of the skin amounts to 
minimizing the area \(\A\big((\oa)_{\bua}\big)\) of the cuts after 
deformation. Therefore, it is optimal to deform elliptical cuts into 
circles, which indeed can be achieved with the swapping pattern. 

Finally, when minimizing the \(L^5\)-norm of the von Mises 
stress \(\M(\Oa)\), we obtained the same zigzag pattern as 
in the stretching in one axial direction. This is reasonable 
since in the present situation it is important to avoid thinning 
of the skin between the cuts where the stress is highest.

The optimal configurations found by the optimization process 
are again summarized in Figure~\ref{fig:configsII} without 
computational noise in the outcomes caused by the 
simulation process.

\section{Conclusion}\label{sct:conclusio}
In this article, we considered the problem of optimizing the
layout of the cut configuration for skin grafting within the framework 
of linear elasticity. The cuts are located in periodic cells and are 
specified by the slope angle. We formulated the optimization problem 
with three objective functionals: the compliance, the \(L^p\)-norm 
of the von Mises stress, and the area of the deformed body. For the
solution of each of these optimization problems, we established 
existence results. By using shape calculus, we derived the 
corresponding shape gradients with respect to the cut 
configuration and implemented a numerical approach based 
on gradient descent. The combination of this approach with
a genetic algorithm avoids getting stuck in one of the numerous 
local minima. We presented reasonable numerical results in the 
case of stretching the skin along one and two axial directions.

The presented approach offers a mathematically solide and 
computationally reasonable tool for the optimization in the 
context of skin grafting. It however also opens several directions 
for future research. An extension of the mathematical model 
to include hyperelastic behaviour would better capture the 
non-linear mechanical properties of the skin. Further, allowing 
more general cut geometries could enhance the practical relevance
of the simulations. Finally, formulating optimization strategies that 
explicitly promote auxetic behaviour may offer promising routes 
for improving the outcomes of skin graft.

\subsection*{Conflict of interest.}
The authors state that there is no conflict of interest.

\subsection*{Replication of results.}
The results presented in this article can be replicated by 
implementing the data structures and algorithms presented 
in this article. The objective functions and their shape gradients 
can be used as described and only require the inputs specified 
in the problem's description.

\nolinenumbers 
\bibliographystyle{plain}
\bibliography{refs}
\end{document}